\documentclass[12pt]{amsart}
\usepackage{amssymb}
\usepackage{amsmath}
\usepackage{amsfonts}

\theoremstyle{plain}
\newtheorem{theorem}{Theorem}[section]
\newtheorem{proposition}[theorem]{Proposition}
\newtheorem{lemma}[theorem]{Lemma}
\newtheorem{corollary}[theorem]{Corollary}

\theoremstyle{remark}
\newtheorem{remark}[theorem]{Remark}

\newcommand{\reft}[1]{Theorem \ref{thm:#1}}
\newcommand{\refp}[1]{Proposition \ref{prop:#1}}
\newcommand{\refl}[1]{Lemma \ref{lem:#1}}
\newcommand{\refps}[2]{Propositions \ref{prop:#1} and \ref{prop:#2}}
\newcommand{\refc}[1]{Corollary \ref{cor:#1}}
\newcommand{\refs}[1]{Section \ref{sec:#1}}

\newcommand{\alp}{\alpha}
\newcommand{\bet}{\beta}
\newcommand{\gam}{\gamma}
\newcommand{\Gam}{\Gamma}
\newcommand{\del}{\delta}
\newcommand{\Del}{\Delta}
\newcommand{\eps}{\varepsilon}
\newcommand{\lam}{\lambda}
\newcommand{\sig}{\sigma}
\newcommand{\tet}{\vartheta}
\newcommand{\vphi}{\varphi}
\newcommand{\zet}{\zeta}

\newcommand{\C}{\mathbb{C}}
\newcommand{\Cx}{\mathbb{C}^{\x}}

\newcommand{\cd}{\mathcal{D}}
\newcommand{\ce}{\mathcal{E}}
\newcommand{\ck}{\mathcal{K}}

\newcommand{\cv}{\mathcal{V}}
\newcommand{\cx}{\mathcal{X}}

\newcommand{\hf}{\hat{f}}

\newcommand{\fq}{\mathbb{F}_{q}}
\newcommand{\fqx}{\fq^{\;\times}}

\newcommand{\ovl}{\overline}
\newcommand{\sset}{\subseteq}
\newcommand{\map}[3]{#1 \colon #2 \to #3}
\newcommand{\frob}[2]{\langle #1 , #2 \rangle}
\newcommand{\set}[2]{\{ #1 \colon #2 \}}

\newcommand{\seq}[2]{ #1_{1}, \ldots, #1_{#2}}
\newcommand{\x}{\times}
\newcommand{\inv}{^{-1}}
\newcommand{\ort}{^{\perp}}

\newcommand{\bcap}{\bigcap}
\newcommand{\bcup}{\bigcup}
\newcommand{\bca}{\begin{cases}}
\newcommand{\eca}{\end{cases}}
\newcommand{\lpar}{\left(}
\newcommand{\rpar}{\right)}

\newcommand{\frg}{{\mathfrak{g}}}
\newcommand{\frh}{{\mathfrak{h}}}
\newcommand{\fru}{{\mathfrak{u}}}
\newcommand{\frv}{{\mathfrak{v}}}

\renewcommand{\sp}{Sp_{2n}(q)}
\newcommand{\eo}{O_{2n}(q)}
\newcommand{\oo}{O_{2n+1}(q)}
\newcommand{\frsp}{\mathfrak{sp}_{2n}(q)}
\newcommand{\freo}{\mathfrak{o}_{2n}(q)}
\newcommand{\froo}{\mathfrak{o}_{2n+1}(q)}

\newcommand{\xa}{\xi_{\alpha,r}}
\newcommand{\la}{\lambda_{\alpha,r}}
\newcommand{\oa}{O^{\ast}_{\alpha,r}}
\newcommand{\xb}{\xi_{\beta,s}}
\newcommand{\lb}{\lam_{\bet,s}}

\newcommand{\xab}{\xa \xb}
\newcommand{\lab}{\la \lb}
\newcommand{\ld}{\lambda_{D,\phi}}
\newcommand{\xd}{\xi_{D,\phi}}
\newcommand{\od}{O^{\ast}_{D,\phi}}
\newcommand{\odd}{O^{\ast}_{D',\phi'}}
\newcommand{\tod}{\Omega^{\ast}_{D,\phi}}
\newcommand{\todd}{\Omega^{\ast}_{D',\phi'}}
\newcommand{\pd}{\map{\phi}{D}{\fqx}}
\newcommand{\cpd}{\map{\vphi}{\cd}{\fqx}}

\newcommand{\xcd}{\zet_{\cd,\vphi}}
\newcommand{\cod}{O^{\ast}_{\cd,\vphi}}
\newcommand{\kd}{K_{\cd,\vphi}}

\newcommand{\xdd}{\xi_{D',\phi'}}
\newcommand{\pdd}{\map{\phi'}{D'}{\fqx}}
\newcommand{\xddd}{\xi_{D'',\phi''}}
\newcommand{\pddd}{\map{\phi''}{D''}{\fqx}}

\newcommand{\zij}{\zet_{i,j,r}}
\newcommand{\mij}{\mu_{i,j,r}}

\newcommand{\irr}{\operatorname{Irr}}
\newcommand{\cf}{\operatorname{cf}}
\newcommand{\scf}{\operatorname{scf}}
\newcommand{\tr}{\operatorname{Tr}}

\newcommand{\all}{\text{ for all }}

\allowdisplaybreaks

\begin{document}


\title[]{Supercharacters of the Sylow $p$-subgroups of the finite symplectic and orthogonal groups }

\author[]{Carlos A. M. Andr\'e \& Ana Margarida Neto}

\address[C. A. M. Andr\'e]{Departamento de Matem\'atica \\ Faculdade de Ci\^encias da Uni\-ver\-si\-da\-de de Lis\-boa \\ Cam\-po Grande \\ Edi\-f\'\i \-cio C6 \\ Piso 2 \\ 1749-016 Lisboa \\ Portugal}

\address[A. M. Neto]{Instituto Superior de Economia e Gest\~ao \\ Universidade T\'ecnica de Lisboa \\ Rua do Quelhas 6 \\ 1200-781 Lisboa \\ Portugal}

\address[C. A. M. Andr\'e \& A. M. Neto]{Centro de Estruturas Lineares e Combinat\'orias \\ Complexo Interdisciplicar da Universidade de Lisboa \\ Av. Prof. Gama Pinto 2 \\ 1649-003 Lisboa \\ Portugal}

\email{caandre@fc.ul.pt}

\email{ananeto@iseg.utl.pt}


\date{April 26, 2008}

\thanks{This research was made within the activities of the Centro de Estruturas Lineares e Combinat\'orias (University of Lisbon, Portugal) and was partially supported by the Funda\c c\~ao para a Ci\^encia e Tecnologia (Lisbon, Portugal) through the project POCTI-ISFL-1-1431. A large part of the research of the first author was made and concluded while he was visiting the University of Stanford (USA) and participating in the special program on ``Combinatorial Representation Theory'' at the MSRI (Berkeley, USA) whose hospitality is gratefully acknowledged, and was partially supported by the sabbatical research grant 4/2008 of the Funda\c c\~ao Luso-Americana para o Desenvolvimento (Lisbon, Portugal). The first author also expresses his sincere gratitude to Persi Diaconis for his invitation to visit the University of Stanford, and for many enlightening discussions regarding supercharacters and their applications.}


\begin{abstract}
We define and study supercharacters of the classical finite unipotent groups of types $B_{n}(q)$, $C_{n}(q)$ and $D_{n}(q)$. We show that the results proved in \cite{AN1} remain valid over any finite field of odd characteristic. In particular, we show how supercharacters for groups of those types can be obtained by restricting the supercharacter theory of the finite unitriangular group, and prove that supercharacters are orthogonal and provide a partition of the set of all irreducible characters. In addition, we prove that the unitary vector space spanned by all the supercharacters is closed under multiplication, and establish a formula for the supercharacter values. As a consequence, we obtain the decomposition of the regular character as an orthogonal linear combination of supercharacters. Finally, we give a combinatorial description of all the irreducible characters of maximum degree in terms of the root system, by showing how they can be obtained as constituents of particular supercharacters.
\end{abstract}

\maketitle


\section{Introduction} \label{sec:intro}

The concept of a ``supercharacter theory'' for an arbitrary finite group was developed by P. Diaconis and I. M. Isaacs in the paper \cite{DI}. Roughly, a supercharacter theory replaces irreducible characters by ``supercharacters'', and conjugacy classes by ``superclasses'', in such a way that a ``supercharacter table'' can be constructed as an ``almost unitary'' matrix with similar properties as the usual character table (namely, orthogonality of rows and columns). More precisely, given any finite group $G$, a {\it supercharacter theory} for $G$ consists of  a partition $\ck$ of $G$ and a set $\cx$ of (complex) characters of $G$ satisfying the following three axioms:
\begin{enumerate}
\item $|\ck| = |\cx|$;
\item every irreducible character of $G$ is a constituent of a unique $\xi \in \cx$;
\item the characters in $\cx$ are constant on the member of $\ck$.
\end{enumerate}
The elements of $\ck$ will be referred to as {\it superclasses}, and the elements of $\cx$ as {\it supercharacters} of $G$. (We observe that, by \cite[Lemma~2.1]{DI}, axiom (ii) is equivalent to requiring that $\{1\} \in \ck$.)

Every finite group $G$ has two ``trivial'' supercharacter theories: the full character theory (where $\cx$ consists of all irreducible characters of $G$, and $\ck$ of all the conjugacy classes of $G$), and the one where $\cx = \{1_{G}, \rho_{G}-1_{G}\}$ and $\ck$ consists of the sets $\{1\}$ and $G - \{1\}$; as usual, we denote by $1_{G}$ the trivial character and by $\rho_{G}$ the regular character of $G$. Although for some groups these are the only possibilities, there are many groups for which nontrivial supercharacter theories exist, and in many cases it may be possible to obtain useful information using some particular supercharacter theory. An illustrating example can be found in the paper \cite{ADS} where E. Arias-Castro, P. Diaconis and R. Stanley showed that a special supercharacter theory can be applied to study a random walk on upper triangular matrices over finite fields using techniques that traditionally required the knowledge of the full character theory.

Supercharacters theories were initially developed for the upper unitrangular group $U_{n}(q)$ consisting of all unipotent upper-triangular $n \times n$ matrices over the finite field $\fq$ with $q$ elements (where $q$ is a power of some prime number $p$). It is known that an explicit description of the irreducible characters and conjugacy classes of $U_{n}(q)$ is an ``intractable'' problem; in fact, in the paper \cite{Gal}, Gudivok et al. show that a ``nice'' description of the conjugacy classes of $U_{n}(q)$ leads to a ``nice'' description of wild quivers. However, in his PhD thesis \cite{A0}, the first author begun the study of the ``basic characters'' of $U_{n}(q)$ (under the assumption that $p \geq n$), and was able to show that ``clumping'' together some of the conjugacy classes and some of the irreducible characters one attains a workable ``approximation'' to the representation theory of $U_{n}(q)$. His results were published in a series of papers in the Journal of Algebra, and showed in particular that the ``basic characters'' determine uniquely the superclasses of a supercharacter theory for $U_{n}(q)$ (the ``theory of basic characters'' was renamed after Roger W. Carter as ``a superclass and supercharacter theory''). We mention that the original theory relies on a construction due to D. Kazdhan (see \cite{K}) and is based on Kirillov's method of coadjoint orbits (see \cite{Ki} for a description of Kirillov's method for the unitriangular group; see also \cite{S1} where J. Sangroniz gives a general version of Kazdhan's construction for algebra groups defined over finite fields of sufficiently large characteristic). Later, in his PhD thesis \cite{Y}, N. Yan showed how the ``basic characters'' can be obtained using more elementary methods which avoid Kazhan's construction and the algebraic geometry involved in it. Yan's approach is valid for an arbitrary prime, and was generalized later by P. Diaconis and M. Isaacs in the paper \cite{DI} in order to extend the theory for an arbitrary finite algebra group defined over $\fq$ (see also \cite{ANi} where a generalization was obtained for algebra groups defined over finite radical rings and over certain rings of $p$-adic integers).

The main goal of this paper is to extend to an arbitrary odd prime the results obtained in the paper \cite{AN1} where the authors started to develop a supercharacter theory for a Sylow $p$-subgroup $U$ of one the (non-twisted) Chevalley groups $C_{n}(q)$, $B_{n}(q)$, and $D_{n}(q)$. (We mention that the present paper is a companion of the forthcoming paper \cite{AN2} where a supercharacter theory for $U$ is established by defining the superclasses of $U$.) As in \cite{AN1}, the notion of a supercharacter of $U$ is very similar the notion of a ``basic character'' of the unitriangular group $U_{n}(q)$, and follows the original idea of parametrizing supercharacters by certain ``minimal'' subsets of (positive) roots. In fact, it is known that the supercharacters of $U_{n}(q)$ can be obtained as certain ``reduced'' products of ``elementary characters'' which are irreducible characters corresponding to the ``matrix entries'' $(i,j)$, for $1 \leq i < j \leq n$, labelled by nonzero elements of $\fq$; in Yan's thesis, the ``elementary characters'' were called ``primary characters'', and the supercharacters were called ``transition characters''. (We refer that the factorization of a supercharacter as a product of ``elementary characters'' holds, not only for the unitriangular group, but also for any finite algebra group, as it is explained in the paper \cite{AO} where O. Pinho and the first author obtain a relation between factorizations of supercharacters and decomposability of certain cyclic modules.)

Following Yan's method, one can show that the supercharacters of $U_{n}(q)$ are parame\-tri\-zed by certain combinatorial data consisting of a ``basic set'' $D$ of matrix entries such that no two elements of $D$ agree in, either the first, or the second, coordinate, and of a map $\phi$ from $D$ to the nonzero elements of $\fq$. (There is an alternative way of parametrizing the supercharacters of $U_{n}(q)$ by labelled set partitions of $\{1, 2, \ldots, n\}$, and we mention that a rich combinatorial structure is arising and appears to have a remarkable analogy with the well-known connection between partitions of $n$ and the representation theory of the symmetric group $S_{n}$; see the papers \cite{T1,T2} by N. Thiem and his collaborators.)

In the present paper, as in \cite{AN1}, we define the supercharacters also as certain ``reduced'' products of ``elementary characters'' (which in general are not necessarily irreducible characters) of the given Sylow $p$-subgroup $U$. These ``reduced'' products are parametrized by pairs consisting of a conveniently chosen ``basic subset of roots'' and of a map to the nonzero elements of $\fq$. (We note that the roots in the unitriangular case are in one-to-one correspondence with the matrix entries.) In fact, the group $U$ can be naturally identified with a subgroup of a unitriangular group, and we will show that the elementary characters (and supercharacters) of $U$ can be obtained as constituents of the restriction of a supercharacter of that unitriangular group.

The paper is organized as follows. In \refs{supchar}, we introduce the necessary notation and define the elementary characters and the supercharacters of the group $U$. Then, in \refs{elem}, we obtain the elementary characters of $U$ by restricting elementary characters of the unitriangular group which contains $U$, and use this information to show that the complex vector space spanned by the supercharacters is, in fact, the associative algebra (finitely generated) by the elementary characters. As a consequence, we deduce that every irreducible character of $U$ is a constituent of a supercharacter. Then, in \refs{ortho}, we prove the orthogonality of supercharacters (as class functions of $U$), by using a partition of the dual space of the Lie algebra of $U$ in terms of its ``basic subvarieties'' as obtained in Theorem~4.5 of the authors' paper \cite{AN1}. In \refs{formula}, we deduce a formula for the supercharacters analogous to the one proved in \cite[Theorem~5.6]{DI}, and obtain a decomposition of the regular character of $U$ as a linear combination (with nonnegative integer coefficients) of supercharacters. Finally, in \refs{max}, we apply our results on supercharacters to identify the irreducible characters of maximum degree of $U$. We observe that, for several times, we refer to results proved for the unitriangular group under the assumption that the prime $p$ is sufficiently large; however, those results are known to be true for arbitrary primes, as it follows from Yan's work (see also \cite{A4}, or \cite{DI}).


\section{Supercharacters} \label{sec:supchar}

Let $p \geq 3$ be a prime number, $q = p^{e}$ ($e \geq 1$) a power of $p$, and $\fq$ the finite field with $q$ elements. For a fixed positive integer $n$, let $G$ denote one of the following classical finite groups: the symplectic group $\sp$, the even orthogonal group $\eo$, or the odd orthogonal group $\oo$ (in alternative notation, these are the (non-twisted) Chevalley groups $C_{n}(q)$, $B_{n}(q)$, and $D_{n}(q)$, respectively). Throughout the paper, we set $U = G \cap U_{m}(q)$ where $$m = \bca 2n, & \text{if $G = \sp$, or $G = \eo$,} \\ 2n+1, & \text{if $G = \oo$,,} \eca$$ and $U_{m}(q)$ denotes the upper unitriangular group consisting of all unipotent upper-triangular $m \times m$ matrices over $\fq$. Then, $U$ is a Sylow $p$-subgroup of $G$, and it is described as follows. Let $J = J_{n}$ be the $n \times n$ matrix with 1's along the anti-diagonal and 0's elsewhere. Then, $U$ consists of all (block) matrices of the form
\begin{equation} \label{e1}
\begin{pmatrix} x & xu & xz \\ 0 & I_{r} & -u^{T}J \\ 0 & 0 & Jx^{-T}J \end{pmatrix}
\end{equation}
where $x \in U_{n}(q)$, $u$ is an $n \x r$ matrix over $\fq$, and
\begin{enumerate}
\item $r = 0$, and $Jz^{T} - zJ = 0$, if $U \leq \sp$;
\item $r = 0$, and $Jz^{T} + zJ = 0$, if $U \leq \eo$;
\item $r = 1$, and $Jz^{T} + zJ = -uu^{T}$, if $U \leq \oo$.
\end{enumerate}

As mentioned in the Introduction, the supercharacters of $U$ will be parametrized by certain subsets of (positive) roots. Thus, we introduce some notation and recall some elementary facts concerning roots; for the details, we refer to the books \cite{C1,C2} by R. Carter (see also \cite[Chapter~8]{CR}). Let $T$ be the maximal torus of $G$ consisting of all diagonal matrices, and $\Sigma$ the root system defined by $T$. The elements of $\Sigma$ are described as follows. For each $1 \leq i \leq n$, let $\map{\eps_{i}}{T}{\fqx}$ be the map defined by $\eps_{i}(t) = t_{i}$ for all $t \in T$; here, we denote by $t_{i} \in \fqx$ the $(i,i)$th entry of the matrix $t \in T$. Then, $\Sigma = \Phi \cup (-\Phi)$ where $$\Phi = \set{\eps_{i} \pm \eps_{j}}{1 \leq i < j \leq n} \cup \Phi'$$ and $$\Phi' = \bca \set{2\eps_{i}}{1 \leq i \leq n}, & \text{if $G = \sp$,} \\ \emptyset, & \text{if $G = \eo$,} \\ \set{\eps_{i}}{1 \leq i \leq n}, & \text{if $G = \oo$.} \eca$$ The roots in $\Phi$ are said to be {\it positive}, and the roots in $-\Phi$ are said to be {\it negative}. Throughout the paper, the word ``root'' will always stand for ``positive root''.

With $\Phi$ we associate the subset of ``matrix entries'' $\ce \sset \set{(i,j)}{-n \leq i, j \leq n}$ as follows. For any $\alp \in \Phi$, we set $$\ce(\alp) = \bca \{(i,j), (-j,-i)\}, & \text{if $\alp = \eps_{i} - \eps_{j}$ for $1 \leq i < j \leq n$}, \\ \{(i,-j), (j,-i)\}, & \text{if $\alp = \eps_{i} + \eps_{j}$ for $1 \leq i < j \leq n$}, \\ \{(i,-i)\}, & \text{if $G = \sp$ and $\alp = 2\eps_{i}$ for $1 \leq i \leq n$,} \\ \{(i,0), (0,-i)\}, & \text{if $G = \oo$ and $\alp = \eps_{i}$ for $1 \leq i \leq n$,} \eca$$ and we define $$\ce = \bcup_{\alp \in \Phi} \ce(\alp).$$ More generally, for each subset $\Psi \sset \Phi$, we set $$\ce(\Psi) = \bcup_{\alp \in \Psi} \ce(\alp);$$ hence, $\ce = \ce(\Phi)$. 

On the other hand, we consider the mirror order $\prec$ on the set $\{0, \pm 1, \ldots, \pm (n+1)\}$ which is defined as $$1 \prec 2 \prec \cdots \prec n+1 \prec 0 \prec -(n+1) \prec \cdots \prec -2 \prec -1,$$ and we shall index the rows (from left to right) and columns (from top to bottom) of any $m \x m$ matrix according to this ordering. Hence, the entries of any matrix $x \in U_{m}(q)$ are indexed by all the pairs $(i,j) \in \ce$: for each $(i,j) \in \ce$, we shall write $x_{i,j}$ to denote the $(i,j)$th entry of $x$ (which occurs in the $i$th row and in the $j$th column). For our purposes, it is convenient to consider the set $$\ce^{+} = \set{(i,j) \in \ce}{1 \leq i \leq n,\ i \prec j \preceq -i},$$ and extend this notation to any subset $\Psi \sset \Phi$ by setting $$\ce^{+}(\Psi) = \ce(\Psi) \cap \ce^{+}.$$ We observe that there exists a one-to-one correspondence between $\Phi$ and $\ce^{+}$.

For any $\alp \in \Phi$, we define the subgroup $U_{\alp}$ of $U$ as follows:
\begin{enumerate}
\item if $\alp = \eps_{i}-\eps_{j}$ for $1 \leq i < j \leq n$, then $$U_{\alp} = \set{x \in U}{x_{i,k} = 0,\  i < k < j};$$
\item if $\alp = \eps_{i}-\eps_{j}$ for $1 \leq i < j \leq n$, then $$U_{\alp} = \set{x \in U}{x_{i,k} = x_{j,l} = 0,\ i < k \leq n,\ j \prec l \preceq 0};$$
\item if, either $\alp = 2\eps_{i}$ for $1 \leq i \leq n$ (in the case where $U \leq \sp$), or $\alp = \eps_{i}$ for $1 \leq i \leq n$ (in the case where $U \leq \oo$), then $$U_{\alp} = \set{x \in U}{x_{i,k} = 0,\ i < k \leq n}.$$
\end{enumerate}

Let $\map{\tet}{\fq}{ \Cx}$ be a non-trivial linear character of the additive group $\fq^{\;+}$ of $\fq$ (this character will be kept fixed throughout the paper; moreover, all characters will be taken over the complex field). For any $r \in \fqx$, the mapping $x \mapsto \tet(rx_{i,j})$ defines a linear character $\map{\la}{U_{\alpha}}{ \Cx}$ of $U_{\alpha}$, and we define the {\it elementary character} $\xa$ to be the induced character $$\xa = (\la)^{U}$$ (see \cite{A1} for the corresponding definition in the case of the unitriangular group; see also \cite[Corollary~5.11]{DI} and the discussion thereon).

We next define the notion of a ``basic subset of roots''. To start with, we recall that a subset $\cd \sset \ce$ is said to be {\it basic} if it contains at  most one entry from each row and at most one root from each column; in other words, $\cd \sset \ce$ is basic if $$|\set{j}{i \prec j \preceq -1,\ (i,j) \in \cd}| \leq 1 \quad \text{and} \quad |\set{i}{1 \preceq i \prec j,\ (i,j) \in \cd}| \leq 1$$ for all $-n \leq i, j \leq n$. Then, we say that $D \sset \Phi$ is a {\it basic subset} if $\cd = \ce(D)$ is a basic subset of $\ce$. (We will always use script letters to denote basic subsets of $\ce$, in contrast to basic subsets of $\Phi$ which will be mostly denoted by italic letters.)

Given any non-empty basic subset $D \sset \Phi$ and any map $\pd$, we define the supercharacter $\xd$ to be the product $$\xd = \prod_{\alp \in D} \xi_{\alp,\phi(\alp)}.$$ For convenience, if $D$ is the empty subset of $\Phi$, we consider the empty map $\pd$, and define $\xd$ to be the unit character $1_{U}$ of $U$. Let $$U_{D} = \bcap_{\alp \in D} U_{\alp} \quad \text{and} \quad \ld = \prod_{\alp \in D} (\lam_{\alp,\phi(\alp)})_{U_{D}}.$$ Then, $\ld$ is clearly a linear character of $U_{D}$ and, by \cite[Proposition~2.2]{AN1}, the supercharacter $\xd$ can be obtained as the induced character
\begin{equation} \label{e2}
\xd = (\ld)^{U}.
\end{equation}

We now state the main result of this paper which extends \cite[Theorem~1.1]{AN1} for arbitrary odd primes. (Given any finite group $G$, we denote by $\irr(G)$ the set of all irreducible characters of $G$, and by $\frob{\cdot}{\cdot}$ (or by $\frob{\cdot}{\cdot}_{G}$ if necessary) the Frobenius' scalar product on the complex vector space of all class functions defined on $G$.)

\begin{theorem} \label{thm:t1}
Let $\chi$ be an arbitrary irreducible character of $U$. Then, $\chi$ is a constituent of a unique supercharacter of $U$; in other words, there exists a unique basic subset $D \sset \Phi$ and a unique map $\pd$ such that $\frob{\chi}{\xd} \neq 0$.
\end{theorem}

Our proof depends strongly on the supercharacter theory of the unitriangular group, and on certain ``basic subvarieties'' defined by polynomial equations on the dual space of the Lie algebra $\fru$ of $U$. We recall its definition. Let $\frg$ denote one of the following classical Lie algebras defined over $\fq$: the symplectic Lie algebra $\frsp$, the even orthogonal Lie algebra $\freo$, or the odd orthogonal Lie algebra $\froo$. Then, $\fru = \frg \cap \fru_{m}(q)$ where $\fru_{m}(q)$ denotes the upper niltriangular Lie algebra consisting of all nilpotent upper-triangular $m \times m$ matrices over $\fq$. Thus, $\fru$ consists of all (block) matrices of the form
\begin{equation} \label{e3}
\begin{pmatrix} a & u & w \\ 0 & 0_{r} & -u^{T}J \\ 0 & 0 & -Ja^{T}J \end{pmatrix}
\end{equation}
where $a \in \fru_{n}(q)$, $u$ is an $n \x r$ matrix over $\fq$, and
\begin{enumerate}
\item $r = 0$, and $Jw^{T} - wJ = 0$, if $\fru \leq \frsp$;
\item $r = 0$, and $Jw^{T} + wJ = 0$, if $\fru \leq \freo$;
\item $r = 1$, and $Jw^{T} + wJ = -uu^{T}$, if $\fru \leq \froo$.
\end{enumerate}
For any $\alp \in \Phi$, we will denote by $e_{\alp}$ the matrix in $\fru$ defined as follows (as usual, $1 \leq i < j \leq n$): $$e_{\alp} = \bca e_{i,j} - e_{-j,-i}, & \text{if $\alp = \eps_{i}-\eps_{j}$,} \\ e_{i,-j} + e_{j,-i}, & \text{if $\alp = \eps_{i}+\eps_{j}$ and $\fru \leq \frsp$,} \\ e_{i,-j} - e_{j,-i}, & \text{if $\alp = \eps_{i}+\eps_{j}$ and $\fru \leq \freo$ or $\fru = \froo$,} \\ e_{i,-i}, & \text{if $\fru \leq \frsp$ and $\alp = 2\eps_{i}$,} \\ e_{i,0} - e_{0,-i}, & \text{if $\fru \leq \froo$ and $\alp = \eps_{i}$.} \eca$$ It is clear that $\set{e_{\alp}}{\alp \in \Phi}$ is an $\fq$-basis of $\fru$. On the other hand, we denote by $\fru^{\ast}$ the dual vector space of $\fru$, and let $\set{e^{\ast}_{\alp}}{\alp \in \Phi}$ be the $\fq$-basis of $\fru^{\ast}$ dual to the basis $\set{e_{\alp}}{\alp \in \Phi}$ of $\fru$; hence, $e^{\ast}_{\alp}(e_{\bet}) = \del_{\alp,\bet}$ for all $\alp, \bet \in \Phi$.

For any $\alp \in \Phi$, we define the Lie subalgebra $\fru_{\alp}$ of $\fru$ as follows:
\begin{enumerate}
\item if $\alp = \eps_{i}-\eps_{j}$ for $1 \leq i < j \leq n$, then $$\fru_{\alp} = \set{a \in \fru}{a_{i,k} = 0,\  i < k < j};$$
\item if $\alp = \eps_{i}+\eps_{j}$ for $1 \leq i < j \leq n$, then $$\fru_{\alp} = \set{a \in \fru}{a_{i,k} = a_{j,l} = 0,\ i < k \leq n,\ j \prec l \preceq 0};$$
\item if, either $\alp = 2\eps_{i}$ for $1 \leq i \leq n$ (in the case where $\fru \leq \frsp$), or $\alp = \eps_{i}$ for $1 \leq i \leq n$ (in the case where $\fru \leq \froo$), then $$\fru_{\alp} = \set{a \in \fru}{a_{i,k} = 0,\ i < k \leq n}.$$
\end{enumerate}
We note that $$\fru_{\alp} =  \sum_{\bet \in \Phi(\alpha)} \fq e_{\bet}$$ where $\Phi(\alp) = \set{\bet \in \Phi}{e_{\beta} \in \fru_{\alp}}$; hence, $\set{e_{\bet}}{\bet \in \Phi(\alp)}$ is a basis of $\fru_{\alp}$.

\begin{remark} \label{rmk:r1}
In the case where $p \geq 2n$, we have $a^{p} = 0$ for all $a \in \fru$, and so we may define the usual exponential map $\map{\exp}{\fru}{U}$ by $\exp(a) = 1 + a + \frac{1}{2!}\, a^{2} + \cdots + \frac{1}{n!}\, a^{n}$ for all $a \in \fru$. It is well-known that $\exp$ is bijective and that the Campbell-Hausdorff formula holds: for all $a, b \in \fru$, we have $\exp(a) \exp(b) = \exp(a+b+\tet(a,b))$ where $\tet(a,b) \in [\fru, \fru]$ (see \cite[pg. 175]{J}). It follows that, if $\frh$ is any Lie subalgebra of $\fru$, then the exponential image $H = \exp(\frh)$ is a subgroup of $U$. In particular, for any $\alp \in \Phi$, we have $U_{\alp} = \exp(\fru_{\alp})$.
\end{remark}


\section{Elementary characters} \label{sec:elem}

We start this section by relating the elementary characters of $U$ with the elementary characters of the corresponding unitriangular group $U_{m}(q)$. Firstly, we fix some notation. To avoid confusion, we shall denote by $\zij$ the elementary character of $U_{m}(q)$ associated with the entry $(i,j) \in \ce$ and the element $r \in \fqx$. By the definition, $\zij$ is the induced character $$\zij = (\mij)^{U_{m}(q)}$$ where $\map{\mij}{U_{i,j}}{\Cx}$ is the linear character of the subgroup $$U_{i,j} = \set{x \in U_{m}(q)}{x_{i,k} = 0,\ i \prec k \prec j}$$ defined by $\mij(x) = \tet(rx_{i,j})$ for all $x \in U_{i,j}$. We observe that, if $1 \preceq i \prec j \preceq 0$ and $\alp = \eps_{i}-\eps_{j}$, then $$U_{i,j} \cap U = U_{\alp} \qquad \text{and} \qquad U_{i,j}U_{\alp} = U_{m}(q);$$ for simplicity of writing, we set $\eps_{0} = 0$. For the remaining cases, the following lemma will be useful.

\begin{lemma} \label{lem:l1}
Let $(i,-j) \in \ce$ for $1 \leq i < j \leq n$ and let $r \in \fqx$. Let $$U'_{i,-j} = \set{x \in U_{m}(q)}{x_{i,b} = x_{-a,-j} = 0,\ i \prec b \preceq 0,\ j \prec a \preceq 0},$$ and let $\map{\nu_{i,-j,r}}{U'_{i,-j}}{\Cx}$ be the linear character of $U'_{i,-j}$ defined by $\nu_{i,-j,r}(x) = \tet(rx_{i,-j})$ for all $x \in U_{i,j}$. Then, $\zet_{i,-j,r} = (\nu_{i,-j,r})^{U_{m}(q)}$.
\end{lemma}

\begin{proof}
For simplicity, we write $H = U_{i,-j}$, $K = U'_{i,-j}$, $\mu = \mu_{i,-j,r}$, $\nu = \nu_{i,-j,r}$ and $\zet = \zet_{i,-j,r}$. It is clear that $\nu$ is a linear character of $K$. By Frobenius' reciprocity, we have $\frob{\zet}{\nu^{U_{m}(q)}} = \frob{\zet_{K}}{\nu}$, whereas, by Mackey's Subgroup Theorem (see \cite[Theorem~17.4(a)]{H}), $$\zet_{K} = (\mu^{U_{m}(q)})_{K} = \sum_{x \in X} (\mu^{x}_{xHx\inv \cap K})^{K}$$ where $X \sset U_{m}(q)$ is a complete set of representatives of the $(H,K)$-double classes of $U_{m}(q)$; without loss of generality, we choose $X$ so that $1 \in X$. Thus, we obtain $$\frob{\zet}{\nu^{U_{m}(q)}} = \sum_{x \in X} \frob{(\mu^{x}_{xHx\inv \cap K})^{K}}{\nu} = \sum_{x \in X} \frob{\mu^{x}}{\nu}_{xHx\inv \cap K}.$$ In particular, for $x = 1$, we get $\frob{\mu}{\nu}_{H \cap K} = 1$ (because, both $\mu$ and $\nu$ are linear), and thus $\frob{\zet}{\nu^{U_{m}(q)}} \neq 0$. Since $\zet$ is irreducible (by \cite[Corollary~5.11]{DI}; see also \cite[Lemma~3]{A1}), we conclude that $\zet$ is an irreducible constituent of $\nu^{U_{m}(q)}$. Since $|U_{m}(q):K| = |U_{m}(q):H|$, we obtain $\zet = \nu^{U_{m}(q)}$ as required.
\end{proof}

We observe that, if $\alp = \eps_{i}+\eps_{j}$ for $1 \leq i < j \leq n$, then $$U'_{i,-j} \cap U = U_{\alp} \qquad \text{and} \qquad U'_{i,-j}U_{\alp} = U_{m}(q).$$ As a consequence of this and the above observations, we may prove the following result.

\begin{proposition} \label{prop:p1}
Let $\alp \in \Phi$, let $(i,j) \in \ce^{+}(\alp)$, and suppose that $j \neq -i$ (in the case where $U \leq \sp$). Then, $(\zij)_{U} = (\zet_{-j,-i,r})_{U} = \xa$ for all $r \in \fqx$. 
\end{proposition}

\begin{proof}
For simplicity, we set $\zet = \zij$. Let $K = U_{i,j}$ if $j \preceq 0$, and $K = U'_{i,j}$ if $0 \prec j \prec -i$. As we observed above, $K \cap U = U_{\alp}$ and $KU = U_{m}(q)$. On the other hand, let $\mu = \mij$ if $j \preceq 0$, $\mu = \nu_{i,j}$ if $0 \prec j \prec -i$. By Mackey's Subgroup Theorem, we get $(\mu^{U_{m}(q)})_{U} = (\mu_{K \cap U})^{U}$. Since $K \cap U = U_{\alp}$ and $\mu_{K \cap U} = \la$, we conclude that $\zet_{U} = \xa$, as required. The proof of the equality $(\zet_{-j,-i,r})_{U} = \xa$ is analogous.
\end{proof}

The previous lemma is not true in the case where $U \leq \sp$ and $\alp = 2\eps_{i}$ for $1 \leq i \leq n$. In fact, we have $U_{\alp} \leq U'_{i,-i}$ (hence, $U'_{i,-i} U_{\alp} = U'_{i,-i} \neq U_{m}(q)$ whenever $i \geq 2$). In order to deal with these cases, we start by proving the following auxiliary result. (The subsets $K_{\cd,\vphi} \sset U_{m}(q)$ are exactly the superclasses of $U_{m}(q)$ as explained in \cite[Appendix~A]{DI}; see also the papers \cite{A3,ADS}, or the PhD thesis \cite{Y}.)

\begin{lemma} \label{lem:l3}
Let  $\cd$ be a basic subset of $\ce$, let $\cpd$ be a map, and let $$e_{\cd,\vphi} = \sum_{(i,j) \in \cd} \vphi(i,j) e_{i,j} \in \fru_{m}(q).$$ Let $O_{\cd,\vphi} = U_{m}(q)e_{\cd,\vphi} U_{m}(q) \sset \fru_{m}(q)$ and $\kd = 1 + O_{\cd,\vphi} \sset U_{m}(q)$. On the other hand, let $$z = \begin{pmatrix} x & xv & xw \\ 0 & I_{r} & -v^{T}J \\ 0 & 0 & J x^{-T} J \end{pmatrix} \in U \qquad \text{and} \qquad a_{z} = \begin{pmatrix} u & v & w \\ 0 & 0_{r} & -v^{T}J \\ 0 & 0 & -J u^{T}J \end{pmatrix} \in \fru$$ where $x = 1+u$. Then, $z \in \kd$ if and only if $a_{z} \in O_{\cd,\vphi}$. Moreover, the mapping $z \mapsto a_{z}$ defines a bijection from $U$ to $\fru$.
\end{lemma}

\begin{proof}
We only consider the case $U \leq \sp$ (the others are similar); hence $r = 0$. Since $$z = \begin{pmatrix} x & 0 \\ 0 & 1 \end{pmatrix} \begin{pmatrix} x & w \\ 0 & Jx^{-T}J \end{pmatrix} \begin{pmatrix} x\inv & 0 \\ 0 & 1 \end{pmatrix}$$ and $\kd$ is invariant under $U_{m}(q)$-conjugation, we conclude that $$z \in \kd \iff  \begin{pmatrix} x & w \\ 0 & J x^{-T} J \end{pmatrix} \in \kd.$$ Since $x\inv - 1 = -u x\inv$, we have $Jx^{-T}J - 1 = (Jx^{-T}J) (-Ju^{T}J)$ and so $$\begin{pmatrix} u & w \\ 0 & J x^{-T}J - 1 \end{pmatrix} = \begin{pmatrix} 1 & 0 \\ 0 & Jx^{-T}J \end{pmatrix} \begin{pmatrix} u & w \\ 0 & -J u^{T}J \end{pmatrix}.$$ It follows that $$a_{z} \in O_{\cd,\vphi} \iff \begin{pmatrix} x & w \\ 0 & J x^{-T} J \end{pmatrix} \in \kd,$$ and this completes the proof.
\end{proof}

We are now able to prove the following result. Given any $\fq$-vector space $V$ and any linear map $f \in V^{\ast}$, we denote by $\tet_{f}$ the composite map $\map{\tet \circ f}{V}{\Cx}$; it is straightforward to check that $\tet_{f}$ is a linear character of the additive group $V^{+}$ and that $\irr(V^{+}) = \set{\tet_{f}}{f \in V^{\ast}}$.

\begin{lemma} \label{lem:l4}
Suppose that $U \leq \sp$, let $\alp = 2\eps_{i}$ for $1 \leq i \leq n$, and let $r \in \fqx$. Then, $\xa$ is an irreducible constituent of $(\zet_{i,-i,r})_{U}$ with multiplicity $1$; in particular, $\xa \neq (\zet_{i,-i,r})_{U}$.
\end{lemma}

\begin{proof}
For simplicity, we write $\xi = \xa$ and $\zet =\zet_{i,-i,r}$. We evaluate the Frobenius' product $\frob{\zet_{U}}{\xi}$. Since $\xi = \lam^{U}$ where $\lam = \la$, we have $\frob{\zet_{U}}{\xi} = \frob{\zet_{U_{\alp}}}{\lam}$. Let $\fru_{2n}(q)^{\ast}$ be the dual space of $\fru_{2n}(q)$, let $e^{\ast}_{i,-i} \in \fru_{2n}(q)^{\ast}$ be defined by $e^{\ast}_{i,-i}(a) = a_{i,-i}$ for all $a \in \fru_{2n}(q)$, and let $O \sset \fru_{2n}(q)^{\ast}$ be the coadjoint $U_{2n}(q)$-orbit which contains $re^{\ast}_{i,-i}$. Then, by \cite[Corollary~5.11]{DI}, we have $$\zet(1+a) = \frac{1}{\sqrt{|O|}} \sum_{f \in O} \tet_{f}(a)$$ for all $a \in \fru_{2n}(q)$; in fact, $O = U_{2n}(q)(re^{\ast}_{i,-i})U_{2n}(q)$.

Let $z \in U_{\alp}$ be arbitrary, and let $a_{z} \in \fru$ be the element defined in the previous lemma; it is clear that $a_{z} \in \fru_{\alp}$. By \cite[Theorem~1]{A2}, there exists a (unique) basic subset $\cd \sset \ce$ and a (unique) map $\cpd$ such that $z \in \kd$. By the previous lemma, we have $1+a_{z} \in \kd$, and so $$\zet(z) = \zet(1+a_{z}) = \frac{1}{\sqrt{|O|}} \sum_{f \in O} \tet_{f}(a_{z}).$$ Since the mapping $z \mapsto a_{z}$ defines a bijection from $U_{\alp}$ to $\fru_{\alp}$, we conclude that $$\frob{\zet_{U_{\alp}}}{\lam} = \frac{1}{\sqrt{|O|}} \sum_{f \in O} \lpar \frac{1}{|\fru_{\alp}|} \sum_{a \in \fru_{\alp}} \tet_{f}(a) \ovl{\tet(r a_{i,-i})} \rpar = \frac{1}{\sqrt{|O|}} \sum_{f \in O} \frob{\tet_{f}}{\tet_{r e^{\ast}_{i,-i}}}_{\fru_{\alp}},$$ and thus $$\frob{\zet_{U_{\alp}}}{\lam} = \frac{|\set{f \in O}{f - r e_{i,-i}^{\ast} \in (\fru_{\alp})\ort}|}{\sqrt{|O|}}$$ where $(\fru_{\alp})\ort = \set{h \in \fru_{2n}(q)^{\ast}}{h(\fru_{\alp}) = 0}$.

Now, we consider the basis $\set{e_{\bet}}{\bet \in \Phi}$ of $\fru$. Let $\Psi = \Phi - \set{\eps_{i}-\eps_{k}}{i < k \leq n}$, so that $\set{e_{\bet}}{\bet \in \Psi}$ is a basis for $\fru_{\alp}$. By \cite[Lemma~1]{A2}, in order to describe $O \cap (\fru_{\alp})^{\perp}$, it is enough to consider vectors $e_{\bet}$ for $\bet \in \Psi$ satisfying $(j,k) \in \ce(\bet)$ for some $i \preceq j \prec k \preceq -i$. Firstly, suppose that $\bet = \eps_{j}-\eps_{k}$ for $i  \leq j < k \leq n$. In this case, we have $e_{\bet} = e_{j,k} - e_{-k,-j}$, and so $h(e_{j,k}) = h(e_{-k,-j})$ for all $h \in (\fru_{\alp})^{\perp}$. Similarly, if $\bet = \eps_{j}+\eps_{k}$ for $i \leq j < k \leq n$, then $e_{\bet} = e_{j,-k} + e_{k,-j}$, and thus $h(e_{j,-k}) = -h(e_{k,-j})$ for all $h \in (\fru_{\alp})^{\perp}$. Finally, if $\bet = 2\eps_{j}$ for $i \leq j \leq n$, then $e_{\bet} = e_{j,-j}$, and so $h(e_{j,-j}) = 0$ for all $h \in (\fru_{\alp})^{\perp}$. Now, suppose that $h = f-re_{i,-i}^{\ast} \in (\fru_{\alp})^{\perp}$ for $f \in O$. Then, since $f(e_{j,-j}) = r\inv f(i,-j) f(e_{j,-i}) = -r\inv f(e_{i,-j})^{2}$ (see \cite[Lemma~1]{A2}), we deduce that $f(e_{i,-j}) = -f(e_{j,-i}) = f(e_{j,-j}) = 0$ for all $i < j \leq n$, and this clearly implies that $$|\set{f \in O}{f - re_{i,-i}^{\ast} \in (\fru_{\alp})^{\perp}}| = q^{2(n - i)}.$$ Since $\sqrt{|O|} = \zet(1) = q^{2(n - i)}$, we conclude that $\frob{\zet_{U}}{\xi} = 1$ as required.
\end{proof}

A similar argument can be used to prove the following result.

\begin{lemma} \label{lem:l5}
Suppose that $U \leq \sp$, let $\alp = 2\eps_{i}$ for $1 \leq i \leq n$, and let $r \in \fqx$. Let $\zet = \zet_{i,-i,r}$ be the supercharacter of $U_{2n}(q)$ associated with $(i,-i)$ and $r$. Moreover, let $\bet = 2\eps_{j}$ for $i \leq j \leq n$, and let $s \in \fqx$ be such that $-rs \in (\fqx)^{2}$. Then, the (irreducible) supercharacter $\xi = \xab$ is a constituent of $\zet_{U}$ with multiplicity $2$.
\end{lemma}

\begin{proof}
By the definition, we have $\xi = \lam^{U}$ where $\lam = \lab$ is the linear character of $V = U_{\alp} \cap U_{\bet}$ defined by $\lam(x) = \tet(rx_{i,-i} + sx_{j,-j})$ for all $x \in V$. Thus, $\frob{\zet_{U}}{\xi}_{U} = \frob{\xi_{V}}{\lam}_{V}$. For each $z \in U$, let $a_{z} \in \fru$ be as in \refl{l3}. Then, for $h = re^{\ast}_{i,-i} + se^{\ast}_{j,-j} \in \fru_{2n}(q)^{\ast}$, we have $\lam(x) = \tet_{h}(a_{x})$ for all $x \in V$. Therefore, since $x \mapsto a_{x}$ defines a bijection from $V$ to $\frv = \fru_{\alp} \cap \fru_{\bet}$, we deduce that $$\frob{\xi_{V}}{\lam}_{V} = \frac{1}{\sqrt{|O|}} \sum_{f \in O} \lpar \frac{1}{|\frv|} \sum_{a \in \frv} \tet_{f}(a) \ovl{\tet_{h}(a)} \rpar = \frac{1}{\sqrt{|O|}} \sum_{f \in O} \frob{\tet_{f}}{\tet_{h}}_{\frv}$$ where $O = U_{2n}(q)(re^{\ast}_{i,-i})U_{2n}(q)$. Thus, we get $$\frob{\xi_{V}}{\lam} = \frac{|\set{f \in O}{f - h \in \frv^{\perp}}|}{\sqrt{|O|}}\;.$$

Now, let $\Psi = \Phi - \set{\eps_{i}-\eps_{k}, \eps_{j}-\eps_{l}}{i < k \leq n,\ j < l \leq n}$, so that $\set{e_{\bet}}{\bet \in \Psi}$ is a basis for $\frv$. As in the previous proof, for an arbitrary element $g \in \frv^{\perp}$, we obtain:
\begin{enumerate}
\item $g(e_{k,l}) = g(e_{-l,-k})$ for all $\bet = \eps_{k}-\eps_{l} \in \Psi$ with $i < k < l \leq n$;
\item $g(e_{k,-l}) = -g(e_{l,-k})$ for all $\bet = \eps_{k}+\eps_{l} \in \Psi$ with $i \leq k < l \leq n$;
\item $g(e_{k,-k}) = 0$ for all $\bet = 2\eps_{k} \in \Psi$ with $i \leq k \leq n$.
\end{enumerate}
Suppose that $f - h \in \frv^{\perp}$ for some $f \in O$. Then, since $f(e_{k,-k}) = r\inv f(i,-k) f(e_{k,-i}) = -r\inv f(e_{i,-k})^{2}$, we deduce that $f(e_{i,-k}) = -f(e_{k,-i}) = f(e_{k,-k}) = 0$ whenever $i \prec k \preceq n$ and $k \neq j$. On the other hand, we have $s = h(e_{j,-j}) = f(e_{j,-j}) = -r\inv f(e_{i,-j})^{2} \in -r\inv (\fqx)^{2}$, and thus $-rs \in (\fqx)^{2}$; moreover, $f(e_{i,-j})$ must be non-zero. Therefore, we conclude that $$|\set{f \in O}{f - h \in \frv^{\perp}}| = 2 q^{2(n-i)},$$ and this completes the proof.
\end{proof}

In the notation of the two previous lemmas, we deduce that $$\xi_{U} = \xa + 2 \sum_{i < j \leq n}\; \sum_{s \sset -r\inv (\fqx)^{2}} \xa \xi_{2\eps_{j},s} + \eta$$ where $\eta$ is, either the zero function, or a character of $U$. Taking degrees, we obtain $$q^{2(n-i)} = q^{n-i} \lpar 1 + 2 \sum_{i < j \leq n} \frac{q-1}{2} q^{n-j} \rpar + \eta(1) = q^{2(n-i)} + \eta(1).$$ It follows that $\eta(1) = 0$, and this concludes the proof of the following result.

\begin{proposition} \label{prop:p2}
Suppose that $U \leq \sp$, let $\alp = 2\eps_{i}$ for $1 \leq i \leq n$, and let $r \in \fqx$. Let $\zet = \zet_{i,-i,r}$ be the supercharacter of $U_{2n}(q)$ associated with $(i,-i)$ and $r$. Then, $$\zet_{U} = \xa + 2 \sum_{i < j \leq n}\; \sum_{s \in -r\inv (\fqx)^{2}} \xa\xi_{2\eps_{j},s}.$$
\end{proposition}


\section{The algebra of superclass functions} \label{sec:super}

The main goal of this section is to prove the existence part of \reft{t1}. We start by proving a result on the decomposition of the product of two elementary characters of $U$. The argument of the proof uses the corresponding decompositions in the case of the unitriangular group $U_{m}(q)$, which can be found in \cite[Lemmas~6-8~and~11]{A1}. (The proofs in that paper use the assumption $p \geq m$, but can be easily adapted for an arbitrary prime (see the thesis \cite{Y} by N. Yan); in fact, every irreducible constituent of any of the given decompositions is a supercharacter of $U_{m}(q)$, and thus is a Kirillov function (see \cite[Theorem~5.10]{DI}, or \cite[Theorem~3]{A5}.)

\begin{proposition} \label{prop:p3}
Let $\alp, \bet \in \Phi$, and $r,s \in \fqx$. Then, the product $\xab$ decomposes as a sum of supercharacters.
\end{proposition}

\begin{proof}
The result is obvious in the case where $\{\alp, \bet\}$ is a basic subset of $\Phi$. Thus, we assume that $\{\alp, \bet\}$ is not basic. Let $(i,j) \in \ce^{+}(\alp)$ and $(k,l) \in \ce^{+}(\bet)$; without loss of generality, we may assume that $i \leq k$. We observe that, by the definition of basic subset of roots, $\ce(\{\alp, \bet\}) = \{(i,j), (-j,-i), (k,l), (-l,-k)\}$ is not a basic subset of $\ce$.  

Firstly, we assume that $j \neq -i$ and $l \neq -k$ (in the case where $U = \sp$). Suppose that $\{(i,j), (k,l)\}$ is a basic subset of $\ce$; hence, we must have $k = -j$. By the previous lemma, we have $\xab = (\zij)_{U} (\zet_{-j,l,s})_{U} = (\zij)_{U} (\zet_{-l,j,s})_{U}$, and thus, by \cite[Lemma~7]{A1}, we obtain $$\xab = (\zij)_{U} + \sum_{-l \prec b \prec j} \sum_{t \in \fqx} (\zij)_{U} (\zet_{-l,b,t})_{U} = \xa + \sum_{-l \prec b \prec j} \sum_{t \in \fqx} \xa (\zet_{-l,b,t})_{U}.$$ The result follows by the \refps{p1}{p2}.

On the other hand, suppose that $\{(i,j), (k,l)\}$ is not a basic subset of $\ce$. Firstly, we consider the case where $i = k$ and $j > l$. By \cite[Lemma~6]{A1}, we obtain $$\xab = \xa + \sum_{i \prec a \prec l} \sum_{t \in \fqx} \xa (\zet_{a,l,t})_{U},$$ and the result follows because the subsets $\{(i,j), (a,l)\}$, for $i \prec a \prec l$, are all basic. The case where $i < k$ and $j = l$ is similar, because $$\xab = \xa + \sum_{k \prec b \prec j} \sum_{t \in \fqx} \xa (\zet_{k,b,t})_{U}$$ (by \cite[Lemma~6]{A1}), and the subsets $\{(i,j), (k,b)\}$, for $k \prec b \prec j$, are all basic. 

Finally, suppose that $i = k$ and $j = l$; hence, $\alp = \bet$. On the one hand, if $s \neq -r$, then $$\xa \xi_{\alp,s} = (1 + (q-1) (j-i-1)) (\zet_{i,j,r+s})_{U} + \sum_{i \prec a, b \prec j-1} \sum_{t \in \fqx} (q-1) (\zet_{i,j,r+s})_{U} (\zet_{a,b,t})_{U}$$ (by \cite[Lemma~11]{A1}), and the result follows as in the previous cases. On the other hand, suppose that $s = -r$. In this case, \cite[Lemma~8]{A1} implies that $$\xa \xi_{\alp,-r} = 1_{U} + \sum_{i \prec a \prec j} \sum_{t \in \fqx} (\zet_{a,j,t})_{U} + \sum_{i \prec b \prec j} \sum_{t' \in \fqx} (\zet_{i,b,t'})_{U} + \sum_{i \prec a, b \prec j} \sum_{t, t' \in \fqx} (\zet_{a,j,t})_{U} (\zet_{i,b,t'})_{U},$$ and thus the result follows by the same reason.

Next, we assume that $U = \sp$ and $l = -k$. Since $\{\alp, \bet\}$ is not basic, the subset $\{(i,j),(k,-k)\}$ cannot be basic; we recall that $\ce(\{\alp, \bet\}) = \{(i,j), (-j,-i), (k,-k)\}$. We start by considering the case where $i \neq k$ and $j = -k$. By \cite[Lemma~6]{A1}, we obtain $$\xa (\zet_{k,-k,s})_{U} = \xa + \sum_{k \prec b \prec -k} \sum_{t \in \fqx} \xa (\zet_{k,b,t})_{U}.$$ Since $\xab$ is a constituent of $\xa (\zet_{k,-k,s})_{U}$ (by \refl{l4}), it is a sum of some of the irreducible constituents of the characters $\xa$ and $\xa (\zet_{k,b,t})_{U}$ for $k \prec b \prec -k$ and $t \in \fqx$; we observe that $\xa (\zet_{k,b,t})_{U} = (\zet_{i,-k,r})_{U} (\zet_{k,b,t})_{U} = (\zet_{k,-i,r})_{U} (\zet_{k,b,t})_{U}$ is reducible (in general). Let $k \prec b \prec -k$ and $t \in \fqx$ be arbitrary. Then, by \cite[Lemma~6]{A1}, we obtain $$\xa (\zet_{k,b,t})_{U} = (\zet_{k,-i,r})_{U} (\zet_{k,b,t})_{U} = \xa + \sum_{k \prec a \prec b} \sum_{t' \in \fqx} \xa (\zet_{a,b,t'})_{U}.$$ Using \refps{p1}{p2}, it is now easy to conclude that the irreducible constituents of $\xa (\zet_{k,b,t})_{U}$ are supercharacters of $U$, and so the result also follows in this situation.

Now, suppose that $(i,j) = (k,-k) = (i,-i)$; hence, $\alp = \bet$. On the one hand, suppose that $s = -r$. Then, by \cite[Lemma~8]{A1}, $(\zet_{i,-i,r})_{U} (\zet_{i,-i,-r})_{U}$ decomposes as a sum with constituents of the form $(\zet_{a,-i,t})_{U} (\zet_{i,b,t'})_{U} = (\zet_{i,a,t})_{U} (\zet_{i,b,t'})_{U}$ for $i \prec a, b \prec -i$ and $t, t' \in \fq$; for simplicity, we set $\zet_{u,v,0} = 1_{U_{m}(q)}$ for all $(u,v) \in \ce$. By the above, each character $(\zet_{i,a,t})_{U} (\zet_{i,b,t'})_{U}$ decomposes as a sum of irreducible supercharacters, and thus $\xa \xi_{\alp,-r}$ also decomposes as a sum of irreducible supercharacters (because $\xa \xi_{\alp,-r}$ is a constituent of $(\zet_{i,-i,r})_{U} (\zet_{i,-i,-r})_{U}$).

On the other hand, suppose that $s \neq -r$. Then, by \cite[Lemma~11]{A1}, $(\zet_{i,-i,r})_{U} (\zet_{i,-i,s})_{U}$ decomposes as a sum with constituents of the form $(\zet_{i,-i,r+s})_{U} (\zet_{a,b,t})_{U}$ for $i \prec a \prec b \prec -i$ and $t \in \fq$. Let $i \prec a \prec b \prec -i$ and $t \in \fqx$; we observe that, by \refp{p1}, we may assume that $1 \leq a \leq n$ and $a \prec b \preceq -a$. By \refp{p2}, the irreducible constituents of $(\zet_{i,-i,r+s})_{U} (\zet_{a,b,t})_{U}$ are also irreducible constituents of characters with the form $\xa (\zet_{c,-c,u})_{U} (\zet_{a,b,t})_{U}$ for some $i \prec c \preceq n$ and some $u \in \fq$. Using a simple inductive argument, we may assume that each character $(\zet_{c,-c,u})_{U} (\zet_{a,b,t})_{U}$ decomposes as a sum of irreducible supercharacters, and this clearly implies that $\xa (\zet_{c,-c,u})_{U} (\zet_{a,b,t})_{U}$ also decomposes as a sum of irreducible supercharacters. The result follows because $\xa \xi_{\alp,s}$ is a constituent of $(\zet_{i,-i,r})_{U} (\zet_{i,-i,s})_{U}$.

Finally, we assume that $i = k$ and $j \prec -k = -i$ (we recall that $i \prec j \preceq -i$). In this case, \cite[Lemma~6]{A1} implies that $$\xa (\zet_{i,-i,s})_{U} = (\zet_{i,-i,s})_{U} + \sum_{i \prec a \prec j} \sum_{t \in \fqx} (\zet_{i,-i,s})_{U} (\zet_{a,j,t})_{U}.$$ As in the previous case, since $\xab$ is a constituent of $\xa (\zet_{i,-i,s})_{U}$, it is a sum of some of the irreducible constituents of the characters $(\zet_{i,-i,s})_{U}$ and $(\zet_{i,-i,s})_{U} (\zet_{a,j,t})_{U}$ for $i \prec a \prec j$ and $t \in \fqx$. Using \refps{p1}{p2}, we easily conclude that the irreducible constituents of $(\zet_{i,-i,s})_{U}$ and $(\zet_{i,-i,s})_{U} (\zet_{a,j,t})_{U}$, for $i \prec a\prec j$, $a \neq -j $, and $t \in \fqx$, are supercharacters of $U$. It remains to consider the irreducible constituents of $(\zet_{i,-i,s})_{U} (\zet_{-j,j,t})_{U}$ for $t \in \fqx$. However, by \refp{p2} and by one of the cases considered previously, all these irreducible constituents are supercharacters.

The proof is complete.
\end{proof}

Next, we prove that the product of supercharacters is a linear combination (with nonnegative integer coefficients) of supercharacters. For the (inductive) proof, we need to endow the set of roots with the total order $\preceq$ defined as follows. Given $\alp, \bet \in \Phi$, we choose $(i,j) \in \ce^{+}(\alp)$ and $(k,l) \in \ce^{+}(\bet)$, and set $\alp \prec \bet$ if and only if, either $l \prec j$, or $j = l$ and $i \prec j$. The following observation will also be very useful (and is an immediate consequence of the previous proof).

\begin{lemma} \label{lem:l0}
Let $\alp, \bet \in \Phi$ be roots with $\alp \preceq \bet$, and let $r,s \in \fqx$. Let $D \sset \Phi$ be a basic subset, and suppose that the supercharacter $\xd$ is a constituent of $\xab$ for some map $\pd$. Let $\gam \in D$ be the smallest root in $D$. Then, $\alp \preceq \gam$, and $\alp \neq \gam$ if and only if $\alp = \bet$ and $s = -r$.
\end{lemma}

We may now proceed with the proof of the following result. Here, we denote by $\cf(U)$ the unitary vector space consisting of all class functions of $U$, and by $\scf(U)$ the vector subspace of $\cf(U)$ spanned by all supercharacters of $U$. (``$\scf$'' stands for ``superclass function''; superclasses of $U$ will be defined in the forthcoming paper \cite{AN2}.)

\begin{theorem} \label{thm:t2}
The product of two supercharacters of $U$ decomposes as a sum of supercharacters. In other words, $\scf(U)$ is a subalgebra of $\cf(U)$; moreover, $\scf(U)$ is finitely generated (as an algebra) by the elementary characters of $U$.
\end{theorem}

\begin{proof}
Let $D, D' \sset \Phi$ be nonempty basic subsets, and let $\xd$ and $\xdd$ be supercharacters of $U$ associated with maps $\pd$ and $\pdd$. Let $\alp \in D'$, and $D'_{0} = D' - \{\alp\}$. By the definition, we have $\xdd = \xa \xi_{D'_{0},\phi'_{0}}$ where $r = \phi'(\alp)$ and $\phi'_{0}$ is the restriction of $\phi'$ to $D'_{0}$, and thus we may use induction on $|D'|$ to conclude that $\xi_{D'_{0},\phi'_{0}} \xd$ decomposes as a sum of supercharacters. Therefore, we are reduced to the case where $\xdd = \xa$; in other words, we must prove that the product $\xa \xd$ decomposes as a sum of supercharacters. To see this, we will proceed by reverse induction on the set of all basic subsets of $\Phi$ endowed with the lexicographic order $\preceq$ which is naturally determined by the total order $\preceq$ on $\Phi$ (as defined above). (We observe that the maximal basic subset of $\Phi$ (with respect to this order) is $\{\eps_{1}-\eps_{2}\}$.)

Let $D$ and $\phi$ be as above. Moreover, let $\alp \in \Phi$ and $r \in \fqx$ be arbitrary, and consider the product $\xa\xd$. Let $\bet \in D$ be the smallest root in $D$, and let $s = \phi(\bet) \in \fqx$. If $D = \{\bet\}$, then $\xa\xd = \xab$ decomposes as a product of supercharacters (by \refp{p3}). Thus, we assume that $|D| \geq 2$. Let $D_{0} = D - \{\bet\}$, and let $\map{\phi_{0}}{D_{0}}{\fqx}$ be the restriction of $\phi$ to $D_{0}$.

Firstly, suppose that $\alp \prec \bet$. Since $D \prec D_{0}$, the product $\xa \xi_{D_{0},\phi_{0}}$ decomposes as a sum of supercharacters (by reverse induction). Let $D'' \sset \Phi$ be a basic subset such that $\xddd$ is a constituent of $\xa \xi_{D_{0},\phi_{0}}$ for some map $\pddd$. Since $\alp \not\in D$ (by the minimal choice of $\bet$), $\alp$ is strictly smaller than all the roots in $D_{0}$. Therefore, \refl{l0} implies that $\alp \in D''$, and so $D \prec D''$. By reverse induction, we conclude that the product $\xb \xddd$ decomposes as a sum of supercharacters, and thus $\xa \xd = \xi_{\bet,\phi(\bet)} (\xa \xi_{D_{0},\phi_{0}})$ also decomposes as a sum of supercharacters.

On the other hand, suppose that $\bet \prec \alp$. As in the previous case, the product $\xa \xi_{D_{0},\phi_{0}}$ decomposes as a sum of supercharacters, and so we may choose a basic subset $D'' \sset \Phi$ such that $\xddd$ is a constituent of $\xa \xi_{D_{0},\phi_{0}}$ for some map $\pddd$. By \refl{l0}, we see that the smallest root in $D_{0} \cup \{\alp\}$ is smaller or equal to the smallest root in $D''$, and so $\bet$ is strictly smaller than the smallest root in $D''$, which means that $D \prec D''$. Thus, as in the previous case, we conclude that the product $\xb \xddd$ decomposes as a sum of supercharacters, and thus $\xa \xd = \xi_{\bet,\phi(\bet)} (\xa \xi_{D_{0},\phi_{0}})$ also decomposes as a sum of supercharacters.

Finally, suppose that $\bet = \alp$. In this case, we have $\xa \xd = (\xa \xi_{\alp,s}) \xi_{D_{0},\phi_{0}}$. By \refp{p3} and \refl{l0}, $\xa \xi_{\alp,s}$ decomposes as a product of supercharacters with the form $\xi_{\alp,t} \xddd$ where $D'' \sset \Phi$ is a basic subset such that $\alp$ is strictly smaller than all of its roots, $\pddd$ is a map, and $t \in \fq$; for simplicity, we write $\xi_{\alp,0} = 1_{U}$. By successively repeating the arguments above, we deduce that $\xddd \xi_{D_{0},\phi_{0}}$ decomposes as a sum of supercharacters, each one corresponding to a basic subset with smallest root larger than $\alp$. Therefore, $D$ is smaller than all of these basic subsets, and so we may use reverse induction to conclude that $\xi_{\alp,t} \xddd \xi_{D_{0},\phi_{0}}$ also decomposes as a sum of supercharacters. It follows that $\xa\xd$ decomposes as a sum of supercharacters, and this completes the proof.
\end{proof}

As a corollary, we obtain the following result.

\begin{corollary} \label{cor:c1}
The restriction $\zet_{U}$ of any supercharacter $\zet$ of $U_{m}(q)$ decomposes as a sum of supercharacters of $U$.
\end{corollary}

\begin{proof}
Since $\zet$ is a product of elementary characters of $U_{m}(q)$, \refps{p1}{p2} imply that $\zet_{U}$ is a product of elementary characters of $U$, and the result follows by the previous theorem.
\end{proof}

We are now able to prove the existence part of \reft{t1}.

\begin{theorem} \label{thm:t3}
Every irreducible character $\chi$ of $U$ is a constituent of a supercharacter of $U$; in other words, there exists a basic subset $D$ of $\Phi$ and a map $\pd$ such that $\frob{\chi}{\xd} \neq 0$.
\end{theorem}

\begin{proof}
Let $\psi$ be an irreducible character of $U_{m}(q)$ with $\frob{\chi}{\psi_{U}} \neq 0$. Let $\cd \sset \ce$ be the (unique) basic subset and $\map{\vphi}{\cd}{\fqx}$ the (unique) map such that $\psi$ is a constituent of the supercharacter $\zet = \zet_{\cd,\vphi}$ of $U_{m}(q)$. Then, $\psi_{U}$ is a constituent of $\zet_{U}$, hence $\chi$ is an irreducible constituent of $\zet_{U}$. The result follows by the previous corollary.
\end{proof}


\section{Orthogonality of supercharacters} \label{sec:ortho}

In this section, we prove the orthogonality of supercharacters, and thus complete the proof of \reft{t1}. The proof depends on the decomposition of $\fru^{\ast}$ as a disjoint union of its ``basic subvarieties'' as defined in the authors' paper \cite{AN1}. We start by recalling their definition.

We fix an arbitrary non-empty basic subset $D \sset \Phi$, and define the $D$-singular and $D$-regular entries as follows. For any $(i,j) \in \ce$, we set $$S(i,j) = \set{(i,k) \in \ce}{k \prec j} \cup \set{(k,j) \in \ce}{i \prec k},$$ and define, for any $\alp \in \Phi$, the subsets $$\ce_{S}(\alp) = \bcup_{(i,j) \in \ce(\alp)} S(i,j) \qquad \text{and} \qquad \ce_{R}(\alp) = \ce - \ce_{S}(\alp)$$ of $\ce$. We say that an entry $(k,l) \in \ce$ is {\it $\alp$-singular} if $(k,l) \in \ce_{S}(\alp)$; otherwise, we say that $(k,l)$ is an {\it $\alp$-regular}. More generally, given an arbitrary basic subset $\cd \sset\ce$, we define $$S(\cd) = \bcup_{(i,j) \in \cd} S(i,j) \qquad \text{and} \qquad R(\cd) = \ce - S(\cd).$$ The entries in $S(\cd)$ are said to be {\it $\cd$-singular}, and the entries in $R(\cd)$ are said to be {\it $\cd$-regular}. We observe that, for any $\alp \in \Phi$, an entry $(k,l) \in \ce$ is $\alp$-singular (resp., $\alp$-regular) if and only if it is $\ce(\alp)$-singular (resp., $\ce(\alp)$-regular). Now, since $D \sset \Phi$ is a basic subset, $\ce(D) = \bcup_{\alp \in D} \ce(\alp)$ is a basic subset of $\ce$ (by the definition). We say that an entry $(i,j) \in \ce$ is {\it $D$-singular} (resp., {\it $D$-regular}) if $(i,j)$ is $\ce(D)$-singular (resp., $\ce(D)$-regular). We denote by $\ce_{S}(D)$ the subset of $\ce$ consisting of all $D$-singular entries, and by $\ce_{R}(D)$ the subset of $\ce$ consisting of all $D$-regular entries. It is clear that $$\ce_{S}(D) = \bigcup_{\alpha \in D} \ce_{S}(\alpha) \qquad \text{and} \qquad \ce_{R}(D) = \ce - \ce_{S}(D).$$ (This definition can be extended to the empty (basic) subset of $\Phi$, in which case all entries  in $\ce$ are regular.)

For an arbitrary basic subset $\cd \sset \ce$ and an arbitrary entry $(i,j) \in \ce$, we denote by $\cd(i,j)$ the subset $$\cd(i,j) = \set{(k,l) \in \cd}{1 \preceq k \prec i,\ j \prec l \preceq -1};$$ it is clear that $\cd(i,j)$ is a basic subset of $\ce$. Let $\cd(i,j) = \{ (i_{1}, j_{1}), \ldots, (i_{t}, j_{t}) \}$ and suppose that $j_{1} \prec \cdots \prec j_{t}$. Moreover, let $\sig \in S_{t}$ be the (unique) permutation such that $i_{\sig(1)} \prec \cdots \prec i_{\sig(t)}$; as usual, we denote by $S_{t}$ the symmetric group of degree $t$. Then, for any $f \in \fru_{m}(q)^{\ast}$, we define $\Del_{i,j}^{\cd}(f) \in \fq$ to be the determinant
\begin{equation} \label{e4}
\Del_{i,j}^{\cd}(f) = \begin{vmatrix} f(e_{i_{\sig(1)}, j}) & f(e_{i_{\sig(1)}, j_{1}}) & \cdots & f(e_{i_{\sig(1)}, j_{t}}) \\ \vdots & \vdots & & \vdots \\ f(e_{i_{\sig(t)}, j}) & f(e_{i_{\sig(t)}, j_{1}}) & \cdots & f(e_{i_{\sig(t)}, j_{t}}) \\ f(e_{i, j}) & f(e_{i, j_{1}}) & \cdots & f(e_{i, j_{t}}) \end{vmatrix}.
\end{equation}
We note that, if $\cd(i,j)$ is empty, then $\Del_{i,j}^{\cd}(f) = f(e_{i,j})$; in particular, if $\cd$ is empty, then $\Del_{i,j}^{\cd}(f) = f(e_{i,j})$ for all $(i,j) \in \ce$.

Now, for any $f \in \fru^{\ast}$, we define the element $u(f) \in \fru$ by $$u(f) = \sum_{\alp \in \Phi} u(f)_{\alp} e_{\alp}$$ where, for each $\alp \in \Phi$, we set $$u(f)_{\alp} = \begin{cases} \frac{1}{2}\, f(e_{\alp}), & \text{if $\alp = \eps_{i}\pm \eps_{j}$ for $1 \leq i < j \leq n$,} \\ f(e_{\alp}), & \text{if $\fru \leq \frsp$ and $\alp = 2\eps_{i}$ for $1 \leq i \leq n$,} \\ \frac{1}{2}\, f(e_{\alp}), & \text{if $\fru \leq \froo$ and $\alp = \eps_{i}$ for $1 \leq i \leq n$.} \end{cases}$$ It is easy to see that $f(v) = \tr(u(f)^{T} v)$ for all $v \in \fru$, and that the mapping $f \mapsto u(f)$ defines a vector space isomorphism from $\fru^{\ast}$ to $\fru$. Finally, we define $\hf \in \fru_{m}(q)^{\ast}$ by $$\hf(v) = \tr(u(f)^{T} v)$$ for all $v \in \fru_{m}(q)$. Then, for any basic subset $D \sset \Phi$ and any entry $(i,j) \in \ce$, we set $$\Del_{i,j}^{D}(f) = \Delta_{i,j}^{\ce(D)}(\hf)$$ for all $f \in \fru$, and, for any map $\pd$, we define the {\it basic subvariety} $$\od = \set{f \in \fru^{\ast}}{\Del_{i,j}^{D}(f) = \Delta_{i,j}^{D}(f_{D,\phi}) \all (i,j) \in \ce_{R}(D)}$$ where
\begin{equation} \label{e5}
f_{D,\phi} = \sum_{\alpha \in D} \phi(\alpha) e^{\ast}_{\alpha} \in \fru^{\ast}.
\end{equation}

The following result is Theorem~4.5 of the paper \cite{AN1} (the proof given there is valid for an arbitrary odd prime).

\begin{theorem} \label{thm:t4}
For any basic subset $D \sset \Phi$ and any map $\pd$, the basic subvariety $\od \sset \fru^{\ast}$ is $U$-invariant (for the usual coadjoint action). Moreover, the vector space $\fru^{\ast}$ decomposes as the disjoint union $$\fru^{\ast} = \bcup_{D,\phi} \od$$ of all its basic subvarieties. 
\end{theorem}

As a particular case, let $\alp \in \Phi$ and $r \in \fqx$. Then, for $D = \{\alp\}$ and $\pd$ defined by $\phi(\alp) = r$, we obtain the {\it elementary subvariety} $$O^{\ast}_{\alp,r} = \od.$$ By \cite[Theorem~5.5]{AN1}, we have
\begin{equation} \label{e6}
\od = \sum_{\alp \in D} O^{\ast}_{\alp,\phi(\alp)}
\end{equation}
for any basic subset $D \sset \Phi$ and any map $\pd$. The elementary subvarieties of $\fru^{\ast}$ determine the elementary supercharacters of $U$ (and vice-versa) as shown by the formula of \refc{c1} below. The following observation will be useful for the proof.

Let $D$ be a basic subset of $\Phi$, and let $\pd$ be a map. Let $f_{D,\phi} \in \fru^{\ast}$ be defined as in \eqref{e5}, and write $\hf_{D,\phi} = \widehat{f_{D,\phi}} \in \fru_{m}(q)^{\ast}$. By the definition, for any $f \in \fru^{\ast}$, $$f \in \od \iff \Del_{i,j}^{\ce(D)}(\hf) = \Delta_{i,j}^{\ce(D)}(\hf_{D,\phi}) \all (i,j) \in \ce_{R}(D).$$ Henceforth, we denote by $\vphi_{D}$ the map $\map{\vphi_{D}}{\ce(D)}{\fqx}$ defined by $\vphi_{D}(i,j) = \hf_{D,\phi}(e_{i,j})$ for all $(i,j) \in \ce(D)$. Let $$O^{\ast}_{\ce(D),\vphi_{D}} = U_{m}(q) \hf_{D,\phi} U_{m}(q) \sset \fru_{m}(q)^{\ast}.$$ By \cite[Propositions~1~and~2]{A2} (see also the discussion in \cite[Appendix~2]{DI}), we know that
\begin{equation} \label{e7}
O^{\ast}_{\ce(D),\vphi_{D}} = \set{f \in \fru_{m}(q)^{\ast}}{\Del_{i,j}^{\cd}(f) = \Delta_{i,j}^{\cd}(\hf_{D,\phi}) \all (i,j) \in \ce_{R}(D)},
\end{equation}
and thus $$O^{\ast}_{D,\phi} = \set{f \in \fru^{\ast}}{\hf \in O^{\ast}_{\ce(D),\vphi_{D}}}.$$

In particular, let $\alp \in \Phi$, $(i,j) \in \ce^{+}(\alp)$ and $r \in \fqx$ be arbitrary. Then, for $D = \{\alp\}$ and $\pd$ defined by $\phi(\alp) = r$, we have $f_{D,\phi} = re^{\ast}_{\alp}$ and $\hf_{D,\phi} = (r/2)(e^{\ast}_{i,j} \pm e^{\ast}_{-j,-i})$ (where the sign is well-determined by $\fru$). By \cite[Proposition~1~and~2]{A2}, we know that $O^{\ast}_{\ce(\alp),\vphi_{\alp}} = O^{\ast}_{i,j,r/2} + O^{\ast}_{-j,-i,\pm r/2}$ where $\vphi_{\alp} = \vphi_{D} = \vphi_{\{\alp\}}$, and thus
\begin{equation} \label{e8}
O^{\ast}_{\alp,r} = \set{f \in \fru^{\ast}}{\hf \in O^{\ast}_{i,j,r/2} + O^{\ast}_{-j,-i,\pm r/2}}.
\end{equation}

We are now able to proceed with the proof of the following result.

\begin{proposition} \label{prop:p5}
Let $\alp \in \Phi$ and let $r \in \fqx$. For any $z \in U$, we denote by $a_{z}$ the element of $\fru$ given by \refl{l3}. Then, $$\xa(z) = \frac{\xa(1)}{|O^{\ast}_{\alp,r}|} \sum_{f \in O^{\ast}_{\alp,r}} \tet_{f}(a_{z})$$ for all $z \in U$.
\end{proposition}

\begin{proof}
Let $(i,j) \in \ce^{+}(\alp)$. Firstly, we consider the case where  $j \neq -i$. By \refp{p2}, we have $\xa = (\zij)_{U}$. Let $O^{\ast}_{i,j,r} = U_{m}(q) (re^{\ast}_{i,j}) U_{m}(q) \sset \fru_{m}(q)^{\ast}$. Then, by \refl{l3} and by \cite[Corollary~5.11]{DI}, we deduce that $$\xa(z) = \zij(z) = \zij(1+a_{z}) = \frac{1}{\sqrt{|O^{\ast}_{i,j,r}|}} \sum_{f \in O^{\ast}_{i,j,r}} \tet_{f}(u_{z})$$ for all $z \in U$.

Let $f \in O^{\ast}_{i,j,r}$ be arbitrary, and consider the restriction $f_{\fru}$ of $f$ to $\fru$. Define $u(f_{\fru}) \in \fru$ and $\hf = \widehat{f_{\fru}} \in \fru_{m}(q)^{\ast}$ as above; hence, $\hf(v) = \tr(u(f_{\fru})^{T} v)$ for all $v \in \fru_{m}(q)$. Let $\map{\vphi}{\ce(\alp)}{\fqx}$ be the map defined by $\vphi(i,j) = u(f)_{i,j} = u(f)_{\alp} = r/2$ and $\vphi(-j,-i) = u(f)_{-j,-i} = \pm r/2$. Using \eqref{e7}, it is straightforward to check that $\hf \in O^{\ast}_{i,j,r/2} + O^{\ast}_{-j,-i,\pm r/2}$, hence $f_{\fru} \in O^{\ast}_{\alp,r}$ (by \eqref{e8}). Since the mapping $f \mapsto f_{\fru}$ clearly defines an injective map from $O^{\ast}_{i,j,r}$ to $\fru$ and since $|O^{\ast}_{\alp,r}| = |O^{\ast}_{i,j,r}|$ (by direct computation), we conclude that $$O^{\ast}_{\alp,r} = \set{f_{\fru}}{f \in O^{\ast}_{i,j,r}}.$$ The result follows because $\xa(1) = \sqrt{|\oa|}$.

On the other hand, suppose that $U \leq \sp$ and $\alp = 2\eps_{i}$ for some $1 \leq i \leq n$. In this case, by \cite[Proposition~3.1 and Theorem~5.5]{AN1}, $\oa \sset \fru^{\ast}$ is the coadjoint $U$-orbit which contains $re^{\ast}_{\alp}$. Let $z \in U$ be fixed. By the definition of induced character, we have $$\xa(z) = (\la)^{U}(z) = \frac{1}{|U_{\alp}|} \sum_{\stackrel{\scriptstyle{x \in U}}{xzx\inv \in U_{\alp}}} \la(xzx\inv).$$ Since $\la(xzx\inv) = \la(xa_{z}x\inv)$ for all $x \in U$ with $xzx\inv \in U_{\alp}$, we deduce that
\begin{align*}
\xa(z) &= \frac{1}{|U_{\alp}|} \sum_{\stackrel{\scriptstyle{x \in U}}{xzx\inv \in U_{\alp}}} \tet_{r e^{\ast}_{\alp}}(xa_{z}x\inv) \\ &= \frac{1}{|U_{\alp}|} \sum_{x \in U} \tet_{re^{\ast}_{\alp}}(xa_{z}x\inv) \lpar \frac{1}{|\fru : \fru_{\alp}|} \sum_{g \in (\fru_{\alp})\ort} \tet_{g}(xa_{z}x\inv) \rpar \\ &= \frac{1}{|U|} \sum_{g \in (\fru_{\alp})\ort} \sum_{x \in U} \tet_{re^{\ast}_{\alp}+g}(xa_{z}x\inv) = \frac{1}{|U|} \sum_{h \in re^{\ast}_{\alp} + (\fru_{\alp})\ort} \sum_{x \in U} \tet_{x \cdot h}(a_{z}).
\end{align*}
Now, it is straightforward to check that $re^{\ast}_{\alp} + (\fru_{\alp})\ort \sset \oa$, and so $$\sum_{x \in U} \tet_{x \cdot h}(a_{z}) = |C_{U}(re^{\ast}_{\alp})| \sum_{f \in \oa} \tet_{f}(a_{z})$$ for all $h \in re^{\ast}_{\alp} + (\fru_{\alp})\ort$; as usual, $C_{U}(re^{\ast}_{\alp})$ denotes the centralizer of $re^{\ast}_{\alp}$ under the coadjoint action of $U$. It follows that $$\xa(z) = \frac{|C_{U}(re^{\ast}_{\alp})|\, |U : U_{\alp}|}{|U|} \sum_{f \in \oa} \tet_{f}(a_{z}) = \frac{|U : U_{\alp}|}{|\oa|} \sum_{f \in \oa} \tet_{f}(a_{z}).$$ The result follows because $\xa(1) = |U:U_{\alp}|$ (by the definition). 
\end{proof}

\begin{remark}
We observe that, by \cite[Proposition~2.1]{AN1}, a basic subvariety $\oa$, for $\alp \in \Phi$ and $r \in \fqx$, is a coadjoint $U$-orbit in all cases, except when $U$ is orthogonal and $\alp = \eps_{i}+\eps_{j}$ for some $1 \leq i, j \leq n$. In this case, by \cite[Theorem~5.5]{AN1}, we have $$\oa = \bcup_{s \in \fq} O_{s}$$ where $\bet = \eps_{i}-\eps_{j}$ and $O_{s} = O(re^{\ast}_{\alp} + se^{\ast}_{\bet})$ denotes the coadjoint $U$-orbit which contains $re^{\ast}_{\alp} + se^{\ast}_{\bet} \in \fru^{\ast}$. On the other hand, an argument similar to the above shows that, for any $s \in \fq$, the expression $$\chi_{s}(z) = \frac{\chi_{s}(1)}{|O_{s}|} \sum_{f \in O_{s}} \tet_{f}(a_{z}), \qquad \text{for all $z \in U$,}$$ defines an irreducible character of $U$. In fact, we may consider the subgroup $$V = \set{z \in U}{a_{z} \in \fru_{\alp} + \fq e_{\bet}}$$ of $U$, and define the linear character $\map{\mu_{s}}{V}{\Cx}$ by $\mu_{s}(z) = \tet(rz_{i,-j}+sz_{i,j})$ for all $z \in V$. Then, we may show that $\chi_{s} = (\mu_{s})^{U}$ for all $s \in \fq$, and that $$\xa = \sum_{s \in \fq} \chi_{s}$$ is the decomposition of $\xa$ into $q$ distinct irreducible constituents (cf. \cite[Proposition~2.1]{AN1}).
\end{remark}

Next, we prove the orthogonality of supercharacters (and thus we conclude the proof of \reft{t1}).

\begin{theorem} \label{thm:t5}
Let $D$ and $D'$ be basic subsets of $\Phi$ and let $\pd$ and $\pdd$ be maps. Then, $\frob{\xd}{\xdd} \neq 0$ if and only if $(D,\phi) = (D',\phi')$.
\end{theorem}

\begin{proof}
Let $z \in U$ be arbitrary, we let $a_{z} \in \fru$ be the element given by \refl{l3}. By the definition, we have $\xd = \prod_{\alp \in D} \xi_{\alp,\phi(\alp)}$, and so (by the previous proposition) $$\xd(z) = \frac{\xd(1)}{|\tod|} \sum_{(\seq{f}{d}) \in \tod} \tet_{f_{1}}(a_{z}) \cdots \tet_{f_{d}}(a_{z}) = \frac{\xd(1)}{|\tod|} \sum_{f \in \od} m_{f} \tet_{f}(a_{z})$$ where $d = |D|$, $\tod = \prod_{\alp \in D} O^{\ast}_{\alp,\phi(\alp)}$, and $$m_{f} = |\set{(\seq{f}{d}) \in \tod}{f_{1} + \cdots + f_{d} = f}|$$ for all $f \in \od$; we recall that $\od = \sum_{\alp \in D} O^{\ast}_{\alp,\phi(\alp)}$ (see \eqref{e6}). Similarly,$$\xdd(z) = \frac{\xdd(1)}{|\todd|} \sum_{g \in \odd} m'_{g} \tet_{g}(a_{z})$$ where $\todd = \prod_{\bet \in D'} O^{\ast}_{\bet,\phi(\bet)}$, and $$m'_{g} = |\set{(\seq{g}{d'}) \in \todd}{g_{1} + \cdots + g_{d'} = g}|$$ for all $g \in \odd$; here, $d' = |D'|$. Since the mapping $z \mapsto a_{z}$ defines a bijection $U \to \fru$, we deduce that
\begin{align*}
\frob{\xd}{\xdd} &= \frac{\xd(1) \xdd(1)}{|\tod|\, |\todd|} \sum_{f \in \od} \sum_{g \in \odd} \lpar \frac{1}{|U|} \sum_{z \in U} \tet_{f}(a_{z}) \ovl{\tet_{g}(a_{z})} \rpar \\ &= \frac{\xd(1) \xdd(1)}{|\tod|\, |\todd|} \sum_{f \in \od} \sum_{g \in \odd} \frob{\tet_{f}}{\tet_{g}}_{\fru} \\ &= \frac{\xd(1) \xdd(1)}{|\tod|\, |\todd|}\; |\od \cap \odd|, 
\end{align*}
and the result follows by \cite[Theorem~4.5]{AN1}.
\end{proof}


\section{A supercharacter formula} \label{sec:formula}

In this section, we establish a formula for the values of a supercharacter $\xi_{D,\phi}$ as a sum over the the basic subvariety $\od \sset \fru^{\ast}$ (which extends \refp{p5}). In fact, we have the following result. 

\begin{theorem} \label{thm:t6}
Let $D$ be a basic subset of $\Phi$, and let $\pd$ be a map. For any $z \in U$, we denote by $a_{z}$ the element of $\fru$ given by \refl{l3}. Then, $$\xd(z) = \frac{\xd(1)}{|\od|} \sum_{f \in \od} \tet_{f}(a_{z})$$ for all $z \in U$.
\end{theorem}

\begin{proof}
Let $z \in U$ be arbitrary. As in the previous proof, we have $$\xd(z) = \frac{\xd(1)}{|\tod|} \sum_{f \in \od} m_{f} \tet_{f}(a_{z})$$ where $\tod = \prod_{\alp \in D} O^{\ast}_{\alp,\phi(\alp)}$, and $$m_{f} = |\set{(f_{\alp})_{\alp \in D} \in \tod}{\sum_{\alp \in D} f_{\alp} = f}|$$ for all $f \in \od$.

Now, let $\cd = \ce^{+}(D)$, let $\map{\vphi}{\cd}{\fqx}$ be the map defined by $\vphi(i,j) = \phi(\alp)$ for all $(i,j) \in \cd$ with $(i,j) \in \ce(\alp)$, and consider the supercharacter $\xcd$ of $U_{m}(q)$. By \refl{l3} and \cite[Theorem~5.6]{DI}, we have $$\xcd(z) = \xcd(1+a_{z}) = \frac{\xcd(1)}{|\od|} \sum_{f \in \cod} \tet_{f}(a_{z}).$$ Let $\map{\pi}{\fru_{m}(q)^{\ast}}{\fru^{\ast}}$ be the natural projection (given by restriction of functions). Since $\pi$ clearly defines an injective map $\map{\pi}{\cod}{\fru^{\ast}}$, we obtain $$\xcd(z) = \frac{\xcd(1)}{|\cod|} \sum_{f \in \pi(\cod)} \tet_{f}(a_{z}).$$ It is straightforward to check that $\od \sset \pi_{\fru}(\cod)$; in fact, we have $\pi(O^{\ast}_{i,j,r}) \sset \oa$ for all $\alp \in \Phi$, all $(i,j) \in \ce(\alp)$ and all $r \in \fqx$ (the equality holds whenever $j \neq -i$; see the proof of \refp{p5}), and the claim follows because $\od = \sum_{\alp \in D} O^{\ast}_{\alp,\phi(\alp)}$ (by \cite[Theorem~5.5]{AN1}). Since $\pi(\cod)$ and $\od$ are $U$-invariant, we conclude that $\pi_{\fru}(\cod)$ decomposes as the disjoint union $$\pi_{\fru}(\cod) = \od \cup \cv$$ for some $U$-invariant subset of $\cv \sset \fru^{\ast}$. Therefore, we get $$\xcd(z) = \frac{\xcd(1)}{|\cod|} \lpar \sum_{f \in \od} \tet_{f}(a_{z}) + \sum_{f \in \cv} \tet_{f}(a_{z}) \rpar.$$

On the other hand, by \refp{p3} and \reft{t1}, we know that $$(\xcd)_{U} = m_{D,\phi} \xd + \eta$$ where $\eta$ is a linear combination (with nonnegative integer coefficients) of supercharacters satisfying $\frob{\xd}{\eta} = 0$. Arguing as in the proof of \reft{t4}, we obtain $$\eta(z) = \sum_{f \in \cv'} n_{f} \tet_{f}(a_{z})$$ for some subset $\cv' \sset \fru^{\ast}$ and some positive integers $n_{f}$ for $f \in \cv'$. Therefore, we get $$\xcd(z) = m_{D,\phi} \xd(z) + \eta(z) = \frac{m_{D,\phi} \xd(1)}{|\tod|} \sum_{f \in \od} m_{f} \tet_{f}(a_{z}) + \sum_{f \in \cv'} n_{f} \tet_{f}(a_{z}).$$ Since $z \in U$ is arbitrary and the mapping $z \mapsto a_{z}$ defines a bijection, we conclude that $$\frac{\xcd(1)}{|\cod|} \lpar \sum_{f \in \od} \tet_{f}(a) + \sum_{f \in \cv} \tet_{f}(a) \rpar = \frac{m_{D,\phi} \xd(1)}{|\tod|} \sum_{f \in \od} m_{f} \tet_{f}(a) + \sum_{f \in \cv'} n_{f} \tet_{f}(a)$$ for all $a \in \fru$, hence $$\frac{\xcd(1)}{|\cod|} \lpar \sum_{f \in \od} \tet_{f} + \sum_{f \in \cv} \tet_{f} \rpar = \frac{m_{D,\phi} \xd(1)}{|\tod|} \sum_{f \in \od} m_{f} \tet_{f} + \sum_{f \in \cv'} n_{f} \tet_{f}$$ (as linear characters of the abelian group $\fru^{+}$). Since $\od \cap \cv = \od \cap \cv' = \emptyset$, we deduce that $$\frac{\xcd(1)}{|\cod|} = \frac{m_{D,\phi} \xd(1) m_{f}}{|\tod|}$$ for all $f \in \od$. Therefore, the coefficients $m_{f}$ do not depend on $f \in \od$, and thus $$\xd(z) = \frac{\xd(1) m}{|\tod|} \sum_{f \in \od} \tet_{f}(a_{z})$$ for a well-determined positive integer $m$. Taking degrees, we obtain $m = |\tod| / |\od|$, and so $$\xd(z) = \frac{\xd(1)}{|\od|} \sum_{f \in \od} \tet_{f}(a_{z})$$ as required.
\end{proof}

As an immediate consequence, we obtain the following result.

\begin{corollary} \label{cor:c2}
Let $D \sset \Phi$ be a basic subset, and let $\pd$ be a map. Then, $\frob{\xd}{\xd} = \xd(1)^{2}/|\od|$; hence, $|\od| = \xd(1)^{2} / \frob{\xd}{\xd}$.
\end{corollary}

\begin{proof}
Using the formula of the previous theorem, we evaluate $$\frob{\xd}{\xd} = \frac{1}{|U|} \sum_{z \in U} \xd(z) \ovl{\xd(z)} = \frac{\xd(1)^{2}}{|\od|^{2}} \sum_{f,g \in \od} \frob{\tet_{f}}{\tet_{g}}_{\fru} = \frac{\xd(1)^{2}}{|\od|},$$ as required.
\end{proof}

Finally, we obtain the following decomposition of the regular character of $U$.

\begin{theorem} \label{thm:t7}
Let $\rho$ be the regular character of $U$. Then, $$\rho = \sum_{D,\phi} \frac{\xd(1)}{\frob{\xd}{\xd}}\; \xd$$ where the sum is over all basic subsets $D \sset \Phi$ and all maps $\pd$.
\end{theorem}

\begin{proof}
Let $z \in U$ be arbitrary. Then,
\begin{align*}
\sum_{D,\phi} \frac{\xd(1)}{\frob{\xd}{\xd}}\; \xd(z) &= \sum_{D,\phi} \frac{\xd(1)}{\frob{\xd}{\xd}} \lpar \frac{\xd(1)}{|\od|} \sum_{f \in \od} \tet_{f}(a_{z}) \rpar \\ &= \frac{\xd(1)^{2}}{\frob{\xd}{\xd}\, |\od|} \sum_{D,\phi}  \sum_{f \in \od} \tet_{f}(a_{z}) \\ &= \sum_{D,\phi}  \sum_{f \in \od} \tet_{f}(a_{z}).
\end{align*}
Since $\fru^{\ast}$ is the disjoint union $\fru^{\ast} = \bcup_{D,\phi} \od$, we obtain $$\sum_{D,\phi}  \sum_{f \in \od} \tet_{f}(a_{z}) = \sum_{f \in \fru^{\ast}} \tet_{f}(a_{z}) = \del_{a_{z},0} |\fru| = \del_{z,1} |U|,$$ and the result follows.
\end{proof}


\section{Irreducible characters of maximum degree} \label{sec:max}

As a final remark, we observe that the description of the irreducible characters of maximum degree of $U$ as given in \cite[Section~6]{AN1} remains valid for arbitrary odd primes. The proofs given there can be adapted (and simplified) using the results of the present paper and also some properties of the Kirillov functions associated with coadjoint $U$-orbits. Given an arbitrary $U$-orbit $O \sset \fru^{\ast}$, we define the {\it Kirillov function} $\map{\phi_{O}}{U}{\C}$ by the rule $$\phi_{O}(z) = \frac{1}{\sqrt{|O|}} \sum_{f \in O} \tet_{f}(a_{z})$$ for all $z \in U$ (see \cite[Section~5]{DI} for the similar definition in the case of finite algebra groups). In fact, it can be shown that every irreducible character of maximum degree is precisely the Kirillov function associated with a (unique) coadjoint $U$-orbit of maximum cardinality. We should mention that similar results have been obtained recently by J. Sangroniz in the paper \cite{S2}, where the author uses Kirillov's method of coadjoint orbits and shows that, for sufficiently large orbits, the associated Kirillov functions are in fact irreducible characters. In this section, we shall use Sangroniz's results to resume the description given in our previous paper \cite{AN1}.

We start by considering the symplectic case $U \leq \sp$. Let $$\Gam = \set{2\eps_{i}}{1 \leq i \leq n} \cup \set{\eps_{i}+\eps_{i+1}}{1 \leq i < n} \sset \Phi.$$ Then, for any basic subset $D \sset \Gam$ and any map $\pd$, the supercharacter $\xd$ is irreducible (by \refc{c2}). In particular, in the case where, either $D$, or $D \cup \{2\eps_{n}\}$, is a maximal basic subset of $\Gam$, then $\xd$ is irreducible and has maximum degree $q^{n(n-1)/2}$ (see the proof of \cite[Proposition~6.3]{AN1}). On the other hand, it is easy to see that the number $d_{n}$ of all these pairs $(D,\phi)$ can be computed by the ``Fibonacci'' recurrence relation $$\bca d_{1} = q,\ d_{2} = q^{2}-1, & \\ d_{n} = (q-1)(d_{n-1}+d_{n-2}), & \text{for $n \geq 3$.} \eca$$ Therefore, by \cite[Theorem~12]{S2}, we obtain the following result.

\begin{theorem} \label{thm:t8}
Let $\chi$ be an irreducible character of $U \leq \sp$. Then, $\chi$ has maximum degree if and only if $\chi = \xd$ where, either $D$, or $D \cup \{2\eps_{n}\}$, is a maximal basic subset of $\Gam$ and $\pd$ is any map.
\end{theorem}

Next, we consider the even orthogonal case $U \leq \eo$. Let $$\Gam = \set{\eps_{i} + \eps_{i+1}}{1 \leq i < n},$$ $D \sset \Gam$ be a basic subset, and $\pd$ be a map. Then, by \refc{c2}, we easily conclude that $\frob{\xd}{\xd} = q^{|D|}$, and a repetition of the proof of \cite[Proposition~6.5]{AN1} shows that $\xd$ is multiplicity free; hence, it has $q^{|D|}$ irreducible constituents each with degree equal to $q^{-|D|} \xd(1)$. In particular, for $$D = \{\eps_{1}+\eps_{2}, \eps_{3}+\eps_{4}, \ldots, \eps_{2r-1}+\eps_{2r}\}$$ where $r = \lfloor n \rfloor$, the supercharacter $\xd$ has $q^{r}$ (distinct) irreducible constituents, each with degree equal to $q^{f(n)}$ where $$f(n) = \bca n(n-2)/2, & \text{if $n$ is even,} \\ (n-1)^{2}/2, & \text{if $n$ is odd.} \eca$$ On the other hand, if $n = 2r$ is even and $$D = \{\eps_{1}+\eps_{2}, \eps_{3}+\eps_{4}, \ldots, \eps_{2r-3}+\eps_{2r-2}\} \sset \Gam,$$ the supercharacter $\xi_{D,\phi}$ has $q^{r-1}$ (distinct) irreducible constituents, each with degree equal to $q^{n(n-2)/2}$. Therefore, for the basic subset $$D = \{\eps_{1}+\eps_{2}, \eps_{3}+\eps_{4}, \ldots, \eps_{2r-3}+\eps_{2r-2}\} \cup \{\eps_{2r-1}-\eps_{2r}\} \sset \Phi$$ and any map $\pd$, the supercharacter $\xd$ also has $q^{r-1}$ (distinct) irreducible constituents, each with degree equal to $q^{n(n-2)/2}$. Now, by \reft{t1}, we may repeat the proof of \cite[Proposition~6.6]{AN1} to conclude that $q^{f(n)}$ is the maximum degree of an irreducible character of $U$, and thus we have obtained $d_{n}$ irreducible characters of maximum degree where $$d_{n} = \bca q^{(n+2)/2} (q-1)^{(n-2)/2}, & \text{if $n$ is even,} \\ q^{(n-1)/2} (q-1)^{(n-1)/2}, & \text{if $n$ is odd.} \eca$$ Using \cite[Theorem~13]{S2}, we conclude the proof of the following result.

\begin{theorem} \label{thm:t9}
Suppose that $U \leq \eo$, and let $\chi$ be an irreducible character of $U$. Let $D \sset \Phi$ be the (unique) basic subset and $\pd$ the (unique) map such that $\frob{\chi}{\xd} \neq 0$. Then, the following holds.
\begin{enumerate}
\item If $n$ is even, then $\chi$ has maximum degree if and only if $$D = \{\eps_{1}+\eps_{2}, \eps_{3}+\eps_{4}, \ldots, \eps_{n-3}+\eps_{n-2}\} \cup D_{1}$$ where $D_{1} \subsetneq \{\eps_{n-1}-\eps_{n}, \eps_{n-1}+\eps_{n}\}$.
\item If $n$ is odd, then $\chi$ has maximum degree if and only if $$D = \{\eps_{1}+\eps_{2}, \eps_{3}+\eps_{4}, \ldots, \eps_{n-1}+\eps_{n}\}.$$
\end{enumerate}
\end{theorem}

Finally, we consider the odd orthogonal case $U \leq \oo$. Let $$\Gam = \set{\eps_{i} + \eps_{i+1}}{1 \leq i < n},$$ let $D \sset \Gam$ be a basic subset, and let $\pd$ be a map. Then, as in the even case, we conclude that, for $$D = \{\eps_{1}+\eps_{2}, \eps_{3}+\eps_{4}, \ldots, \eps_{2r-1}+\eps_{2r}\} \sset \Gam$$ where $r = \lfloor n \rfloor$, the supercharacter $\xd$ has $q^{r}$ (distinct) irreducible constituents, each with degree equal to $q^{n(n-1)/2}$. On the other hand, using \refc{c2} (see also \cite[pg.~423]{AN1}), we conclude that, for the basic subset $$D = \bca \{\eps_{1}+\eps_{2}, \eps_{3}+\eps_{4}, \ldots, \eps_{n-3}+\eps_{n-2}\} \cup \{\eps_{n-1}\}, & \text{if $n = 2r$ is even,} \\ \{\eps_{1}+\eps_{2}, \eps_{3}+\eps_{4}, \ldots, \eps_{n-2}+\eps_{n-1}\} \cup \{\eps_{n}\}, & \text{if $n = 2r+1$ is odd,} \eca$$ and any map $\pd$, the supercharacter $\xd$ has, either $q^{r -1}$, or $q^{r}$, (distinct) irreducible constituents, each with degree equal to $q^{n(n-1)/2}$. Finally, as in the even case, we conclude that $q^{n(n-1)/2}$ is the maximum degree of an irreducible character of $U$, and thus we have obtained $d_{n}$ irreducible characters of maximum degree where $$d_{n} = \bca q^{(n-2)/2} (q+1) (q-1)^{n/2}, & \text{if $n$ is even,} \\ q^{(n+1)/2} (q-1)^{(n-1)/2}, & \text{if $n$ is odd.} \eca$$ Using \cite[Theorem~13]{S2}, we conclude the proof of the following result.

\begin{theorem} \label{thm:t10}
Suppose that $U \leq \oo$, and let $\chi$ be an irreducible character of $U$. Let $D \sset \Phi$ be the (unique) basic subset and $\pd$ the (unique) map such that $\frob{\chi}{\xd} \neq 0$. Then, the following holds.
\begin{enumerate}
\item If $n$ is even, then $\chi$ has maximum degree if and only if $$D = \{\eps_{1}+\eps_{2}, \eps_{3}+\eps_{4}, \ldots, \eps_{n-3}+\eps_{n-2}\} \cup D_{1}$$ where, either $D_{1} = \{\eps_{n-1}+\eps_{n}\}$, or $D_{1} = \{\eps_{n}\}$.
\item If $n$ is odd, then $\chi$ has maximum degree if and only if $$D = \{\eps_{1}+\eps_{2}, \eps_{3}+\eps_{4}, \ldots, \eps_{n-2}+\eps_{n-1}\} \cup D_{1}$$ where $D_{1} \sset \{\eps_{n}\}$.
\end{enumerate}
\end{theorem}


\end{document}